\documentclass{siamltex}

\usepackage{latexsym,amssymb,amsmath,enumerate,bbm,graphicx,psfrag,color}
\usepackage[latin1]{inputenc}

\def\norm#1{\hspace{0.2ex} \|#1\| \hspace{0.2ex}} 

\newcommand{\labeq}[1]{\label{eq:#1}}			
\def\req#1{{\rm(\ref{eq:#1})}}

\usepackage[ocgcolorlinks, linkcolor=blue]{hyperref}

\newcommand{\R}{\ensuremath{\mathbbm{R}}} 
\newcommand{\C}{\ensuremath{\mathbbm{C}}} 
\newcommand{\N}{\ensuremath{\mathbbm{N}}} 
\newcommand{\range}{\mathcal R}
\newcommand{\kernel}{\mathcal N}
\newcommand{\dx}[1][x]{\ensuremath{\,{\rm{d}} #1}}

\newcommand{\kommentar}[1]{}

\renewcommand{\Re}{\operatorname{Re}}
\newcommand{\spn}{\operatorname{span}}

\newtheorem{example}[theorem]{Example}
\newtheorem{remark}[theorem]{Remark}

\begin{document}

\title{Monotonicity and local uniqueness for the Helmholtz equation}
       
\author{Bastian Harrach\footnotemark[2], Valter Pohjola\footnotemark[3], and Mikko Salo\footnotemark[4]}
\renewcommand{\thefootnote}{\fnsymbol{footnote}}

\maketitle

\footnotetext[2]{Institute for Mathematics, Goethe-University Frankfurt, Frankfurt am Main, 
Germany (\url{harrach@math.uni-frankfurt.de})}

\footnotetext[3]{Research Unit of Mathematical Sciences, University of Oulu, Oulu, 
Finland (\url{valter.pohjola@gmail.com})}

\footnotetext[4]{Department of Mathematics and Statistics, University of Jyväskylä, Jyväskylä, 
Finland (\url{mikko.j.salo@jyu.fi})}

\let\thefootnote\relax\footnotetext{\hrule \vspace{1ex} \centering First published in \emph{Analysis \& PDE} in volume \textbf{12}(7), 1741--1771, 2019,
published by Mathematical Sciences Publishers (\url{https://doi.org/10.2140/apde.2019.12.1741}).
}

\pagestyle{myheadings}
\thispagestyle{plain}
\markboth{Bastian Harrach, Valter Pohjola, and Mikko Salo}{Monotonicity and local uniqueness for the Helmholtz equation}


\begin{abstract}
This work extends monotonicity-based methods in inverse problems to the case of the Helmholtz (or stationary Schr\"odinger) equation $(\Delta + k^2 q) u = 0$ in a bounded domain for fixed
non-resonance frequency $k>0$ and real-valued scattering coefficient function $q$. 
We show a monotonicity relation between the scattering coefficient $q$ and the local Neumann-Dirichlet operator 
that holds up to finitely many eigenvalues. Combining this with the method of localized potentials, or Runge approximation, adapted to the case where finitely many constraints are present, we derive a constructive monotonicity-based characterization of scatterers from partial boundary data. We also obtain the local uniqueness result that two coefficient functions $q_1$ and $q_2$ can be distinguished by partial boundary data if there is a neighborhood of the boundary part where $q_1\geq q_2$ and $q_1\not\equiv q_2$. 
\end{abstract}



\section{Introduction}
\label{Sec:intro}

Let $\Omega\subseteq \R^n$, $n\geq 2$, be a bounded Lipschitz domain with unit outer normal $\nu$. 
For a fixed non-resonance frequency $k>0$, we study the relation between a real-valued scattering coefficient function $q\in L^\infty(\Omega)$ in the Helmholtz equation (or time independent Schr\"odinger equation)
\begin{equation}\labeq{Intro:Helmholtz}
(\Delta + k^2 q) u = 0 \quad \text{ in } \Omega
\end{equation}
and the local (or partial) Neumann-to-Dirichlet (NtD) operator
\[
\Lambda(q):\ L^2(\Sigma)\to L^2(\Sigma), \quad g\mapsto u|_\Sigma,
\]
where $u\in H^1(\Omega)$ solves \req{Intro:Helmholtz} with Neumann data  
\[
\partial_\nu u|_{\partial \Omega}=\left\{ \begin{array}{l l} g & \text{ on } \Sigma,\\ 0 & \text{ else.}\end{array}\right.
\]
Here $\Sigma\subseteq \partial \Omega$ is assumed to be an arbitrary non-empty
relatively open subset of $\partial \Omega$. Since $k$ is a non-resonance
frequency, $\Lambda(q)$ is well defined and is easily shown to be a
self-adjoint compact operator.

We will show that
\[
q_1\leq q_2\quad \text{ implies } \quad \Lambda(q_1)\leq_\text{fin} \Lambda(q_2),
\]
where the inequality on the left hand side is to be understood pointwise almost everywhere, and
the right hand side denotes that $\Lambda(q_2)-\Lambda(q_1)$ possesses only finitely many negative eigenvalues.
Based on a slightly stronger quantitative version of this monotonicity relation, and an extension of the technique of localized potentials \cite{gebauer2008localized} to spaces with finite codimension, we deduce the following local uniqueness result for determining $q$ from $\Lambda(q)$.

\begin{theorem}\label{intro:thm:local_uniqueness}
Let $O\subseteq \overline\Omega$ be a connected relatively open set with $O\cap \Sigma\neq \emptyset$ and
$q_1\leq q_2$ on $O$. Then 
\[
\Lambda(q_1) = \Lambda(q_2) \quad \text{ implies } \quad q_1 = q_2 \text{ in $O$}.
\]
Moreover, if $q_1|_O \not \equiv q_2|_O$, then $\Lambda(q_2)-\Lambda(q_1)$ has infinitely many positive eigenvalues.
\end{theorem}

Theorem~\ref{intro:thm:local_uniqueness} will be proven in section
\ref{Sec:Local}. Note that this result removes the assumption $q_1,q_2\in L^\infty_+(\Omega)$
from the local uniqueness result in \cite{harrach2017local}, and that it implies global uniqueness if
$q_1-q_2$ is piecewise-analytic, cf.\ corollary \ref{cor:pcw_anal_uniqueness}.
Note also that in dimension $n=2$, Imanuvilov, Uhlmann and Yamamoto \cite{imanuvilov2015neumann} have proven global uniqueness 
with partial boundary data for potentials $q\in W^{1,p}(\Omega)$, $p>2$. Compared to the result in \cite{imanuvilov2015neumann}, theorem~\ref{intro:thm:local_uniqueness} is both, less restrictive as it holds for $L^\infty$-potentials and any dimension $n\geq 2$,
and more restrictive as it relies on a local definiteness condition that is not required in \cite{imanuvilov2015neumann}.

Additionally to theorem \ref{intro:thm:local_uniqueness}, we will also derive a constructive monotonicity-based method to detect a
scatterer in an otherwise homogeneous domain. Let the scatterer $D\subseteq
\Omega$ be an open set such that $\overline D\subseteq \Omega$ and the
complement $\Omega\setminus \overline D$ is connected, and let
\begin{alignat*}{2}
q(x)&=1 \quad && \text{ for } x\in \Omega\setminus D \text{ (a.e.)}, \text{ and }\\
1<  q_\text{min}\leq  q(x)& \leq   q_\text{max} \quad && \text{ for } x\in D \text{ (a.e.)},
\end{alignat*}
with constants $q_\text{min},q_\text{max}>0$. For an open set $B\subseteq \Omega$, we define the self-adjoint compact operator
\[
T_B:\ L^2(\Sigma)\to L^2(\Sigma), \quad
\int_{\Sigma} g T_B h\dx[s]:=\int_B k^2 u^{(g)}_1 u^{(h)}_1\dx.
\]
where $u^{(g)}_1,u^{(h)}_1\in H^1(\Omega)$ solve \req{Intro:Helmholtz} with $q\equiv 1$ and Neumann data $g$, resp. $h$.

\begin{theorem}\label{intro:thm:main1} For all $0<\alpha\leq q_\text{min}-1$,
\[
B\subseteq D \quad \text{ if and only if } \quad \alpha T_B \leq_{\mathrm{fin}} \Lambda(q)-\Lambda(1).
\]
\end{theorem}

We will also give a bound on the number of negative eigenvalues in the case $B\subseteq D$, and prove a similar result 
for scatterers with negative contrast in section \ref{Sec:Detection}.

Let us give some references on related works and comment on the origins and relevance of our
result. The inverse problem considered in this work is closely related to
the inverse conductivity problem of determining the positive conductivity function $\gamma$ in 
the equation $\nabla \cdot (\gamma \nabla u) = 0$ in a bounded domain
in $\R^n$ from knowledge of the associated Neumann-Dirichlet operator. 
This is also known as the problem of Electrical Impedance Tomography or the
Calder\'on Problem \cite{calderon1980inverse,calderon2006inverse}.
For a short
list of seminal contributions for full boundary data let us refer to
\cite{kohn1984determining,kohn1985determining,druskin1998uniqueness,sylvester1987global,nachman1996global,astala2006calderon,haberman2013uniqueness,caro2016global}.
For the uniqueness problem with partial boundary data there are rather
precise results if $n=2$ (see \cite{IUY,imanuvilov2015neumann} and the survey \cite{GT_survey}), but
in dimensions $n \geq 3$ it is an open question whether measurements on an
arbitrary open set $\Sigma \subseteq \partial \Omega$ suffice to determine the
unknown coefficient. We refer to \cite{KSU, Isakov, KS, krupchyk2016calderon}
and the overview article \cite{kenig2014recent} for known results, which either
impose strong geometric restrictions on the inaccessible part of the boundary
or require measurements of Dirichlet and Neumann data on sets that cover 
a neighborhood of the so-called front face
\[
 F(x_0)=\{ x\in \partial \Omega:\ (x-x_0)\cdot \nu(x)\leq 0\} 
\]
for a point $x_0$ outside the closed convex hull of $\Omega$. Also note that partial boundary data 
determines full boundary data by unique continuation if there exists a connected neighborhood of the full boundary on which the coefficient is known,
so that uniqueness also holds in this case, cf.\ \cite{AmmariUhlmann}. 

Theorem \ref{intro:thm:local_uniqueness}, as well as the previous work
\cite{harrach2017local}, give uniqueness results where the measurements are
made on an arbitrary open set $\Sigma \subseteq \partial \Omega$. Our result shows
that a coefficient change in the positive or negative direction in a neighborhood of $\Sigma$ (or an open subset of $\Sigma$) always leads to a change in the Neumann-Dirichlet-operator irrespectively of what happens outside this neighborhood, or the geometry or topology of the domain. Note however that our uniqueness result requires that there is a neighborhood of the
boundary part on which the coefficient change is of definite sign. Our uniqueness result does not cover coefficient changes
that are infinitely oscillating between positive and negative values when approaching the boundary.

Our result is based on combining monotonicity estimates (similar to those originally derived in \cite{kang1997inverse,ikehata1998size}) with localized potentials. Other theoretical uniqueness results have been obtained by this approach in \cite{arnold2013unique,gebauer2008localized,harrach2009uniqueness,harrach2010exact,harrach2012simultaneous,harrach2017local}. 
Also note that monotonicity relations have been used in various ways in the study of inverse problems, 
see, e.g., \cite{kohn1984determining,kohn1985determining,isakov1988,alessandrini1990,ikehata1999identification},
where uniqueness results are established by methods that involve monotonicity conditions and blow-up arguments.

Monotonicity-based methods for detecting regions (or inclusions) where a coefficient function differs from a known background
have been introduced by Tamburrino and Rubinacci \cite{tamburrino2002new} for the inverse conductivity problem.
\cite{tamburrino2002new} proposed to simulate boundary measurements for a number of test regions and then use the fact that a monotonicity relation between the simulated and the true measurements will hold, if the test region lies inside the true inclusion. The work \cite{harrach2013monotonicity} used the technique of localized potentials \cite{gebauer2008localized} to prove that this is really an if-and-only-if-relation for the case of continuous measurements modeled by the NtD operator.
Moreover, \cite{harrach2013monotonicity} also showed that this if-and-only-if-relation still holds 
when the simulated measurements are replaced by linearized approximations so that the monotonicity method 
can be implemented without solving any forward problems other that that for the known background medium. 
For a list of recent works on monotonicity-based methods, let us refer to 
\cite{harrach2015combining,harrach2015resolution,harrach2016enhancing,maffucci2016novel,tamburrino2016monotonicity,barth2017detecting,garde2017comparison,garde2017convergence,garde2017regularized,harrach2017monotonicity,su2017monotonicity,ventre2017design,brander2018monotonicity,griesmaier2018monotonicity,harrach2018monotonicity,zhou2018monotonicity,harrach2018global,harrach2018uniqueness,harrach2019dimension,harrach2019fractional_II}. 

Previous monotonicity-based results often considered second order
equations with positive bilinear forms, such as the conductivity equation. 
So far, this positivity has been the key to proving monotonicity inequalities between 
the coefficient and the Neumann-to-Dirichlet operator, and previous results fail to hold
in general for equations involving a positive frequency $k > 0$ (or a negative potential for the Schrödinger equation). 
In this article, we remove this limitation and introduce methods for more general elliptic models.
We will focus on the Helmholtz equation in a bounded domain as a model case, but the
ideas might be applicable to inverse boundary value and scattering problems
for, e.g., Helmholtz, Maxwell, and elasticity equations.
The main technical novelty of this work is that we treat compact perturbations of positive bilinear forms 
by extending the monotonicity relations to only hold true up to finitely many eigenvalues, and to extend the localized potentials arguments to hold on spaces of finite codimension. 

It should also be noted that the localized potentials arguments in \cite{gebauer2008localized} stem from the ideas of the Factorization Method that was originally developed for scattering problems involving far-field measurements of the Helmholtz equation by Kirsch \cite{Kir98}, cf.\ also the book of Kirsch and Grinberg \cite{kirsch2008factorization}, and then extended to the inverse conductivity problem by Brühl and Hanke \cite{Bru00,Bru01}, cf.\ also the overview article \cite{harrach2013recent}. For the inverse conductivity problem, the Monotonicity Method has the advantage over the Factorization Method that it allows a convergent regularized numerical implementation (cf.\ \cite[Remark~3.5]{harrach2013monotonicity} and \cite{garde2017regularized}) and that it can also be used for the indefinite case where anomalies of larger and smaller conductivity are present. 
The localized potentials approach in \cite{gebauer2008localized} has recently been extended to show the possibility of localizing and concentrating electromagnetic fields in \cite{harrach2018localizing}.

The paper is structured  as follows. In section \ref{Sec:Helmholtz} we 
discuss the well-posedness of the Helmholtz equation outside resonance frequencies, 
introduce the Neu\-mann-to-Dirichlet-operators, and give a unique continuation result from sets of positive measure.
Section~\ref{Sec:Monotonicity} and \ref{sect:loc_pot} contain the main theoretical tools for this work. 
In section~\ref{Sec:Monotonicity}, we introduce a Loewner order of compact self-adjoint operators that holds up to finitely many negative eigenvalues, and show that increasing the scattering index monotonically increases the Neumann-to-Dirichlet-operator in the sense of this new order. We also characterize the connection between the finite number of negative eigenvalues that have to be excluded in the Loewner ordering and the Neumann eigenvalues for the Laplacian. Section \ref{sect:loc_pot}
extends the localized potentials result from \cite{gebauer2008localized} to the Helmholtz equation and shows that the energy terms appearing in the monotonicity relation can be controlled in spaces of finite codimension. We give two independent proofs of this result, one using a functional analytic relation between operator norms and the ranges of their adjoints, and an alternative proof that is based on a Runge approximation argument.
Section \ref{Sec:Local} and \ref{Sec:Detection} then contain the main results of this work on local uniqueness for the bounded Helmholtz equation and the detection of scatterers by monotonicity comparisons, cf.\ theorem \ref{intro:thm:local_uniqueness}
and \ref{intro:thm:main1} above.

A preliminary version of these results has been published as the extended abstract \cite{harrach2017oberwolfach}. The bound on the number of negative eigenvalues in the monotonicity inequalities derived in this work has recently been improved in \cite{harrach2019dimension}.

\noindent {\bf Acknowledgements.}
V.P.\ and M.S.\ were supported by the Academy of Finland (Finnish Centre of
Excellence in Inverse Problems Research, grant number 284715) and by an ERC
Starting Grant (grant number 307023).

\section{The Helmholtz equation in a bounded domain}
\label{Sec:Helmholtz}

We start by summarizing some properties of the Neumann-to-Dirichlet-operators, discuss well-posedness and the role of resonance frequencies, and state a unique continuation result for the Helmholtz equation in a bounded domain.

\subsection{Neumann-to-Dirichlet-operators}
\label{Subsec:NtD}

Throughout this work, let $\Omega \subseteq \R^{n}$, $n\geq 2$, denote a bounded domain with Lipschitz boundary and outer unit normal $\nu$, and let 
$\Sigma\subseteq \partial \Omega$ be an open subset of $\partial \Omega$. 
For a frequency $k\geq 0$ and a real-valued scattering coefficient function $q\in
L^\infty(\Omega)$, we consider the Helmholtz equation with (partial) Neumann boundary
data $g\in L^2(\Sigma)$, i.e., to find $u\in H^1(\Omega)$ with
\begin{equation}\labeq{Helmholtz}
(\Delta + k^2 q) u = 0 \quad \text{ in } \Omega, \qquad
\partial_\nu u|_{\partial \Omega}=\left\{ \begin{array}{l l} g & \text{ on } \Sigma,\\ 0 & \text{ else.}\end{array}\right.
\end{equation}
We also denote the solution with $u^{(g)}_q$ instead of $u$ if the choice of $g$ and $q$ is not
clear from the context. 

The Neumann problem \req{Helmholtz} is equivalent to the variational
formulation of finding $u\in H^1(\Omega)$ such that
\begin{equation}\labeq{Varform}
\int_\Omega \left( \nabla u \cdot \nabla v - k^2 q  u v\right) \dx[x]=
\int_{\partial \Omega} g v|_{\partial \Omega} \dx[s]
\quad \text{ for all } v\in H^1(\Omega).
\end{equation}
We introduce the bounded linear operators
\begin{align*}
    I: & \ H^1(\Omega) \to H^1(\Omega),\\
    j: &\ H^1(\Omega)\to L^2(\Omega),\\
    M_{q}: & \ L^2(\Omega)\to L^2(\Omega),
\end{align*}
where $I$ denotes the identity operator, $j$ is the compact embedding from $H^1$ to $L^2$,
and $M_{q}$ is the multiplication operator by $q$. We furthermore 
use $\langle \cdot, \cdot \rangle$ to denote the $H^1(\Omega)$
inner product and define the operators 
$$
K:=j^* j, \quad\text{ and }\quad  K_q:=j^* M_{q} j,
$$
which are compact self-adjoint linear operators from $H^1(\Omega)$ to $H^1(\Omega)$. By
\[
\gamma_{\Sigma}:\ H^1(\Omega) \to L^2(\Sigma),\quad v\mapsto v|_{\Sigma}
\]
we denote the compact trace operator. 

With this notation \req{Varform} can be written as
\[
\langle (I-K-k^2 K_q)u,v\rangle=\int_{\partial \Omega} g (\gamma_{\Sigma}v) \dx[s] \quad \text{ for all } v\in H^1(\Omega),
\]
so that the Neumann problem for the Helmholtz equation \req{Helmholtz} is equivalent to the equation
\begin{equation}\labeq{Helmholtz_Operatorform}
(I-K-k^2 K_q) u = \gamma_{\Sigma}^* g.
\end{equation}

Our results on identifying the scattering coefficient $q$ will require that $I-K-k^2 K_q$ is continuously invertible, which is equivalent to 
the fact that \emph{$k$ is not a resonance frequency}, or, equivalently, that \emph{$0$ is not a Neumann eigenvalue}, 
see lemma~\ref{lemma:resonances} and lemma~\ref{lemma:Neumann_EV}.
Note that this implies, in particular, that $k>0$ and $q\not\equiv 0$. We can then define the Neumann-to-Dirichlet operator (with Neumann data prescribed and Dirichlet data measured on the same open subset $\Sigma \subseteq \partial \Omega$)
\[
\Lambda(q):\ L^2(\Sigma)\to L^2(\Sigma), \quad g\mapsto u|_{\Sigma}, \text{ where } u\in H^1(\Omega) \text{ solves \req{Helmholtz}.}
\]
The Neumann-to-Dirichlet operator fulfills
\begin{equation}\labeq{NtD_Operatorform}
\Lambda(q)= \gamma_\Sigma (I-K-k^2 K_q)^{-1} \gamma_\Sigma^*,
\end{equation}
which shows that $\Lambda(q)$ is a compact self-adjoint linear operator.

We will show in section \ref{Sec:Monotonicity}, that there is a monotonicity relation between 
the scattering coefficient $q$ and the Neumann-to-Dirichlet-operator $\Lambda(q)$. Increasing $q$ will increase $\Lambda(q)$ in the sense of operator definiteness up to finitely many eigenvalues. The number of eigenvalues that do not follow the increase will be bounded by the number 
defined in the following lemma. Note that here, and throughout the paper, we always count the number of eigenvalues of a compact self-adjoint operator with multiplicity according to the dimension of the associated eigenspaces.


\begin{lemma}\label{lemma:d_q}
Given $k>0$, and $q\in L^\infty(\Omega)$, let $d(q)$ be the number of eigenvalues of $K+k^2 K_q$ that are larger than $1$, and 
let $V(q)$ be the sum of the associated eigenspaces. Then $d(q)=\dim(V(q))\in \N_0$ is finite, and
\[
\int_\Omega \left(|\nabla v|^2 - k^2 q |v|^2\right)\dx \geq 0 \quad \text{ for all $v\in V(q)^\perp$}
\]
where $V(q)^{\perp}$ denotes the orthocomplement of $V(q)$ in $H^1(\Omega)$.
\end{lemma}
\begin{proof}
Since
\[
\langle (I-K-k^2 K_q)v,v\rangle=\int_\Omega \left(|\nabla v|^2 - k^2 q |v|^2\right) \dx,
\]
the assertion follows from the spectral theorem for compact self-adjoint operators.
\end{proof}

We will show in lemma~\ref{lemma:Neumann_EV} that $d(q)$ agrees with the number of positive Neumann eigenvalues of $\Delta+k^2q$.
If $q(x)\leq q_\text{max}\in \R$ for all $x\in \Omega$ (a.e.) then $d(q)\leq d(q_\text{max})$, and $d(q_\text{max})$ is the number of 
Neumann eigenvalues of the Laplacian $\Delta$ that are larger than $-k^2q_\text{max}$, cf.\ corollary \ref{corollary:d_vs_Neumann_EV_Laplacian}.

\subsection{Resonance frequencies}\label{subsect:resonances}

We now summarize some results on the solvability of the Helmholtz equation \req{Helmholtz} outside of
resonance frequencies.

\begin{lemma}\label{lemma:resonances}
Let $q \in L^{\infty}(\Omega)$.
\begin{enumerate}[(a)]
\item For each $k\geq 0$, the following properties are equivalent:

\begin{enumerate}[(i)]
\item For each $F\in L^2(\Omega)$ and $g\in L^2(\partial \Omega)$, there exists a unique solution $u\in H^1(\Omega)$
of 
\begin{equation}\labeq{Helmholtz_with_rhs}
(\Delta + k^2 q) u = F \quad \text{ in } \Omega, \quad \partial_\nu u|_{\partial \Omega}= g,
\end{equation}
and the solution depends linearly and continuously on $F$ and $g$. 
%
\item The homogeneous Neumann problem 
\begin{equation}\labeq{Helmholtz_homogeneous}
(\Delta + k^2 q) u = 0 \quad \text{ in } \Omega, \quad \partial_\nu u|_{\partial \Omega}= 0,
\end{equation}
admits only the trivial solution $u\equiv 0$.
\item The operator $I-K-k^2 K_q:\ H^1(\Omega)\to H^1(\Omega)$ is continuously invertible.
\end{enumerate}

$k$ is called a \emph{resonance frequency}, if the properties (i)--(iii) do not hold.

\item If $q\not\equiv 0$, then the set of resonance frequencies is countable and discrete.
\end{enumerate}
\end{lemma}
\begin{proof}
\begin{enumerate}[(a)]
\item Clearly, (i) implies (ii), and, using the equivalence of \req{Helmholtz} and \req{Helmholtz_Operatorform}, (ii) implies that $I-K-k^2 K_q$ is injective. Since $K$ and $K_q$ are compact, the operator $I-K-k^2 K_q$ is Fredholm of index 0. Hence, injectivity of $I-K-k^2 K_q$ already implies that $I-K-k^2 K_q$ is continuously invertible, so that (ii) implies (iii). Finally, $u\in H^1(\Omega)$ solves
\req{Helmholtz_with_rhs} if and only if
\begin{equation*}
\int_\Omega \left( \nabla u \cdot \nabla v - k^2 q u v\right) \dx[x]=
-\int_\Omega F v \dx + \int_{\partial \Omega} g v|_{\partial \Omega} \dx[S]
\quad \text{ for all } v\in H^1(\Omega).
\end{equation*}
This is equivalent to 
\[
\langle (I-K-k^2 K_q)u,v\rangle=-\int_\Omega F j(v) \dx + \int_{\partial \Omega} g \gamma_{\partial \Omega}(v) \dx[s] \quad \text{ for all } v\in H^1(\Omega),
\]
and thus equivalent to
\[
(I-K-k^2 K_q)u= -j^*F + \gamma_{\partial \Omega}^* g,
\]
so that (iii) implies (i).
\item We extend $I$, $K$, and $K_q$ to the Sobolev space of complex valued functions
\[
I,K,K_q:\ H^1(\Omega;\C)\to H^1(\Omega;\C).
\]
For $k\in \C$ we then define
\[
R(k):= K + k^2 K_q:\ H^1(\Omega;\C)\to H^1(\Omega;\C).
\]
$R(k)$ is a family of compact operators depending analytically on $k\in \C$. 
The analytic Fredholm theorem (see, e.g., \cite[Thm.\ VI.14]{RSI}) now implies
that either $I-R(k)$ is not invertible for all $k\in \C$, or that 
there is a countable discrete set $Z \subseteq \C$ such that $I-R(k)$
is continuously invertible when $k \in \C \setminus Z$. Hence, to prove (b), it suffices to
show that there exists $k\in \C$ for which $I-R(k)$ is invertible.

We will show that this is the case for any $0\neq k\in \C$ with $\Re (k^2) = 0$. In fact,
$(I - R(k))u = 0$ implies that
\begin{equation*} 
0 = \int_{\Omega} \left( \nabla u\cdot \nabla v - k^2 q u v \right) \dx \quad \text{ for all } v\in H^1(\Omega;\C).
\end{equation*}
Using $v:=\overline{u}$ and taking the real part yields that $0 = \int_{\Omega} |\nabla u|^2 \dx$,
which shows that $u$ must be constant, and that
\[
\int_\Omega k^2 q u v \dx=0\quad \text{ for all } v\in H^1(\Omega;\C).
\]
Together with $k^2\neq 0$, and $q\not\equiv 0$, this shows that $u\equiv 0$.
Hence, $I - R(k)$ is injective and thus invertible for all $0\neq k\in \C$ with $\Re (k^2) = 0$.
\end{enumerate}
\end{proof}

\subsection{Unique continuation}
\label{Subsec:UCP}

We will make use of a unique continuation property for the
Helmholtz equation from sets of positive measure. 
In two dimensions, this follows from a standard reduction to quasiconformal mappings.
However, since we could not find a proof in the literature we will first give the argument following \cite{Alessandrini_ucp} and references therein (in fact \cite{Alessandrini_ucp} proves strong unique continuation for more general equations). See also \cite{AstalaIwaniecMartin} for basic facts on quasiconformal mappings in the plane.

\begin{lemma}\label{lemma:UCP_2d}
Let $\Omega \subset \R^2$ be a connected open set, and suppose that $u \in
H^1_{\mathrm{loc}}(\Omega)$ is a weak solution of 
\[
-\mathrm{div}(A\nabla u) + du = 0 \text{ in $\Omega$},
\]
where $A \in L^{\infty}(\Omega, \R^{n \times n})$ is symmetric and satisfies
$A(x) \xi \cdot \xi \geq c_0 |\xi|^2$ for some $c_0 > 0$, and $d \in
L^{q/2}(\Omega)$ for some $q > 2$. If $u$ vanishes in a set $E$ of positive
measure, then $u \equiv 0$ in $\Omega$.
\end{lemma}
\begin{proof}
It is enough to show that $u$ vanishes in some ball, since then weak (or strong) unique continuation \cite{Alessandrini_ucp} implies that $u \equiv 0$.  Let $x_0$ be a point of density one in $E$ and let $U_r := B_r(x_0)$ and $E_r := E \cap U_r$. There is $r_0 > 0$ so that if $r < r_0$, then $U_r \subset \Omega$ and $E_r$ has positive measure.

We will now work in $U_r$. Observe first that there is $p > 2$ so that $u \in W^{1,p}(U_r)$ \cite[Theorem 16.1.4]{AstalaIwaniecMartin}. In particular $u$ is H\"older continuous and we may assume (after removing a set of
measure zero from $E$) that $u(x) = 0$ for all $x \in E_r$. The first step is to show that $\nabla u = 0$ a.e.\ on $E_r$. Let $N_1$ be the set of points in $E_r$ where $u$ is not differentiable, and let $N_2$ be the set of points of density $< 1$ in $E_r$. Then $N_1$ and $N_2$ have zero measure. Fix a point $x \in E_r \setminus (N_1 \cup N_2)$ and a unit direction $e$. There is a sequence $(x_j)$ with $x_j \in B(x,1/j) \cap E_r$ so that $|\frac{x_j-x}{|x_j-x|} - e| \leq 1/j$ for $j$ large (for if not, then all points in $E_r$ near $x$ would be outside a fixed sector in direction $e$ which contradicts the fact that $x$ has density one). Since $u$ is differentiable at $x$, 
\[
u(x_j) - u(x) = \nabla u(x) \cdot (x_j-x) + o(|x-x_j|).
\]
Dividing by $|x-x_j|$ and using that $u(x_j) = u(x) = 0$ implies that $\nabla u(x) \cdot e = 0$. It follows that $\nabla u$ vanishes in $E_r \setminus (N_1 \cup N_2)$, so indeed 
\begin{equation} \label{u_vanishing}
u = 0 \text{ in $E_r$}, \qquad \nabla u = 0 \text{ a.e.\ in $E_r$}.
\end{equation}

The next step is to reduce to the case where $d=0$. As in \cite[Proposition 2.4]{Alessandrini_ucp}, we choose $r$ small enough so that there is a nonvanishing $w \in W^{1,p}(U_r)$ satisfying  
\begin{gather*}
-\mathrm{div}(A\nabla w) + dw = 0 \text{ in $U_r$}, \\
1/2 \leq w \leq 2 \text{ in $U_r$}, \quad \norm{\nabla w}_{L^p(U_r)} \leq 1.
\end{gather*}
We write $v = u/w$. It follows that $v \in W^{1,p}(U_r)$ is a weak solution of 
\[
-\mathrm{div}(\widetilde{A} \nabla v) = 0 \text{ in $U_r$}
\]
where $\widetilde{A} = w^2 A$ is $L^{\infty}$ and uniformly elliptic. Moreover, \eqref{u_vanishing} implies that 
\begin{equation} \label{v_vanishing}
v = 0 \text{ in $E_r$}, \qquad \nabla v = 0 \text{ a.e.\ in $E_r$}.
\end{equation}
To prove the lemma, we will show that $v \equiv 0$ in some ball.

Let $J = \left( \begin{array}{cc} 0 & -1 \\ 1 & 0 \end{array} \right)$. Since $\widetilde{A} \nabla v$ is divergence free, there is a real valued function $\tilde{v} \in H^1(U_r)$ satisfying 
\begin{equation} \label{v_gradient}
\nabla \tilde{v} = J(\widetilde{A} \nabla v).
\end{equation}
Such a function $\tilde{v}$ is unique up to an additive constant. Define 
\[
f = v + i \tilde{v}.
\]
As in \cite{Alessandrini_ucp}, $f \in H^1(U_r)$ solves an equation of the form 
\[
\partial_{\bar{z}} f = \mu \partial_z f + \nu \overline{\partial_z} f \text{ in $U_r$}
\]
where $\norm{\mu}_{L^{\infty}(U_r)} + \norm{\nu}_{L^{\infty}(U_r)} < 1$. It follows that $f$ is a quasiregular map and by the Stoilow factorization \cite[Theorem 5.5.1]{AstalaIwaniecMartin} it has the representation 
\[
f(z) = F(\chi(z)), \qquad z \in U,
\]
where $\chi$ is a quasiconformal map $\C \to \C$ and $F$ is a holomorphic function on $\chi(U)$.

Finally, the Jacobian determinant $J_f$ of $f$ is given by 
\[
J_f(z) = F'(\chi(z)) J_{\chi}(z).
\]
Using \eqref{v_vanishing} and \eqref{v_gradient}, we see that $J_f = 0$ a.e.\ in $E_r$. Moreover, since $\chi$ is quasiconformal, $J_{\chi}$ can only vanish in a set of measure zero \cite[Corollary 3.7.6]{AstalaIwaniecMartin}. It follows that $F'(\chi(z)) = 0$ for a.e.\ $z \in E_r$. Then the Taylor series of the analytic function $F'$ at $\chi(x_0)$ must vanish (otherwise one would have $F'(\chi(z)) = (\chi(z)-\chi(x_0))^N g(\chi(z))$ where $g(\chi(x_0)) \neq 0$ and the only zero near $x_0$ would be $z=x_0$). Thus $F' = 0$ near $x_0$, so $F$ is constant, $f$ is also constant, and $v = 0$ near $x_0$.
\end{proof}

We can now state the unique continuation property for any dimension $n\geq 2$ in the form that we will utilize in 
the later sections. As in \cite[Def.~2.2]{harrach2013monotonicity} we say that
a relatively open subset $O\subseteq \overline \Omega$ is \emph{connected to
$\Sigma$} if $O$ is connected and $\Sigma\cap O\neq \emptyset$.

\begin{theorem}\label{thm:UCP}
\begin{enumerate}[(a)]
\item Let $u\in H^1(\Omega)$ solve
\begin{equation}\labeq{UCP_Helmholtz}
    (\Delta +k^2 q) u = 0\quad \text{ in $\Omega$.} 
\end{equation}
If $u|_E=0$ for a subset $E\subseteq \Omega$ with positive measure then $u(x)=0$ for all $x\in \Omega$ (a.e.)
\item Let $u\in H^1(\Omega)$, $\Delta u\in L^2(\Omega)$, and 
\begin{equation*}
    (\Delta +k^2 q) u = 0 \quad \text{ in $\Omega\setminus C$,}
\end{equation*}
for a closed set $C$ for which $\overline \Omega\setminus C$ is connected to
$\Sigma$.
If $u|_{\Sigma}=0$ and $\partial_\nu u|_{\Sigma}=0$, then $u(x)=0$ for all $x\in \Omega\setminus C$ (a.e.)
\end{enumerate}
\end{theorem}
\begin{proof}
For $n=2$, (a) follows from lemma~\ref{lemma:UCP_2d}. For $n\geq 3$, (a) is shown in \cite[Theorem~4.2]{harrach2017local} (see also \cite[proof of Theorem 2.1]{regbaoui2001unique}) by combining the following two results:
 \begin{enumerate}
\item[(i)] If $u\in H^1(\Omega)$ solves \req{UCP_Helmholtz} and vanishes on a measurable set of positive measure 
then $u$ has a zero of infinite order (see, e.g., Figueiredo and Gossez \cite[Proposition 3]{de1992strict}, or Hadi and Tsouli \cite[Theorem 2.1]{hadi2001strong}).
\item[(ii)] The trivial solution $u\equiv 0$ is the only $H^1(\Omega)$-solution of \req{UCP_Helmholtz} that has a zero of infinite order (see, e.g, the book of H\"ormander \cite[Theorem 17.2.6]{hormander1983analysis}).
\end{enumerate}

(b) follows from (a) by extending $u$ by zero on $B\setminus \Omega$ where $B$ is a small ball 
with $B\cap \partial \Omega\subseteq \Sigma$, cf.\ the proof of lemma 4.4c) in \cite{harrach2017local}.
\end{proof}

\section{Monotonicity and localized potentials for the Helmholtz equation}
\label{Sec:Monotonicity}

In this section we show that increasing the scattering coefficient leads to a
larger Neumann-to-Dirichlet operator in a certain sense. For this result, the Neumann-to-Dirichlet operators are ordered by an extension of the Loewner
order of compact self-adjoint operators that holds up to finitely many negative eigenvalues.

\subsection{A Loewner order up to finitely many eigenvalues}\label{subsect:Loewner_order}

We start by giving a rigorous definition and characterization of this ordering. 

\begin{definition}
Let $A,B:\ X\to X$ be two self-adjoint compact linear operators on a Hilbert
space $X$. For a number $d\in \N_0$, we write 
\[
A\leq_{d} B, \quad \text{ or } \quad \langle Ax,x\rangle \leq_d \langle Bx,x\rangle, 
\]
if $B-A$ has at most $d$ negative eigenvalues. We also write $A\leq_\text{fin}
B$ if $A\leq_{d} B$ holds for some $d\in \N_0$, and we write $A\leq B$ if
$A\leq_{d} B$ holds for $d=0$.   
\end{definition}

Note that for $d=0$ this is the standard partial ordering of compact
self-adjoint operators in the sense of operator definiteness (also called
Loewner order).  Also note that ''$\leq_\text{fin}$'' and ''$\leq_{d}$'' (for
$d\neq 0$) are not partial orders since they are clearly not antisymmetric. Obviously,
''$\leq_\text{fin}$'' and ''$\leq_{d}$'' are reflexive, and 
''$\leq_\text{fin}$'' is also transitive (see
lemma~\ref{lemma:order_transitive} below) and thus a so-called preorder.

To characterize this new ordering, we will make use of the following lemma.
\begin{lemma}\label{lemma_eigenvalues_quadform}
Let $A:\ X\to X$ be a self-adjoint compact linear operator on a Hilbert space $X$ with inner product $\langle \cdot, \cdot\rangle$ 
inducing the norm $\norm{\cdot}$. Let $d\in \N_0$ and $r\in \R$, $r\geq 0$.
\begin{enumerate}[(a)]
\item The following statements are equivalent:
\begin{enumerate}[(i)]
\item $A$ has at most $d$ eigenvalues larger than $r$.
\item There exists a compact self-adjoint operator $F:\ X\to X$ with 
\[
\dim( \range(F))\leq d,\quad \text{ and } \quad \langle (A - F) x,x\rangle \leq
        r \norm{x}^2 \quad \text{ for all } x\in X,
\]
where $\range(F)$ stands for the range of $F$.
\item There exists a subspace $W\subset X$ with $\mathrm{codim}(W)\leq d$ such that
\[
\langle Aw,w\rangle \leq r \norm{w}^2 \quad \text{ for all } w\in W.
\] 
\item There exists a subspace $V\subset X$ with $\dim (V)\leq d$ such that
\[
\langle Av,v\rangle \leq r \norm{v}^2 \quad \text{ for all } v\in V^\perp.
\]
\end{enumerate}
\item The following statements are equivalent:
\begin{enumerate}[(i)]
\item $A$ has (at least) $d$ eigenvalues larger than $r$.
\item There exists a subspace $V\subset X$ with $\dim (V)\geq d$ such that
\[
\langle Av,v\rangle > r \norm{v}^2 \quad \text{ for all } v\in V.
\]
\end{enumerate}
\end{enumerate}
\end{lemma}
\begin{proof}
\begin{enumerate}[(a)]
\item We start by showing that (i) implies (ii). Let $A$ have at most $d$
eigenvalues larger than $r\geq 0$.
Let $(\lambda_k)_{k\in \N}$ be the non-zero eigenvalues of $A$, ordered in such
a way that $\lambda_k\leq r$ for $k>d$. 
Let $\kernel(A)$ denote the kernel of $A$ and let $(v_k)_{k\in \N}\in X$ be a 
sequence of corresponding eigenvectors forming an orthonormal basis of $\kernel(A)^\perp$. Then
\[
Ax=\sum_{k=1}^\infty \lambda_k v_k \langle v_k, x\rangle \quad \text{ for all } x\in X,
\]
and (ii) follows with $F:\ X\to X$ defined by
\[
F:\ x\mapsto \sum_{k=1}^d \lambda_k v_k \langle v_k, x\rangle \quad \text{ for all } x\in X.
\]

The implication from (ii) to (iii) follows by setting $W:=\kernel(F)$ since 
\[
\mathrm{codim}(W)=\dim(W^\perp)=\dim(\range(F))\leq d
\]
and 
\[
\langle Aw,w\rangle = \langle (A - F)w,w\rangle \geq 0.
\]

(iii) implies (iv) by setting $V:=W^\perp$.

To show that (iv) implies (i), we assume that (i) is not true, so that $A$ has at least $d+1$ eigenvalues larger than $r\geq 0$.
We sort the positive eigenvalues of $A$ in decreasing order to obtain
\[
\lambda_1\geq \cdots \geq \lambda_d\geq \lambda_{d+1}>r.
\]
Then, by the Courant-Fischer-Weyl min-max principle, (see, e.g. \cite[p.\ 318]{PL})
we have that the minimum over all $d$-dimensional subspaces $V\subset X$ must fulfill

\[
\min_{V\subset X \atop \dim(V)=d} \max_{\ v\in V^\perp \atop \norm{v}=1} \langle A v,v\rangle=\lambda_{d+1}>r,
\]
which shows that (iv) cannot be true. Hence, (iv) implies (i). 
\item can be shown analogously to (a). (ii) follows from (i) by choosing $V$ as the sum of eigenspaces for eigenvalues larger than $r$, 
and (ii) implies (i) by using the Courant-Fischer-Weyl min-max principle. 
\end{enumerate}
\end{proof}

\begin{corollary}\label{cor_order_characterization}
Let $A,B:\ X\to X$ be two self-adjoint compact linear operators on a Hilbert space $X$ with inner product $\langle \cdot, \cdot\rangle$. For any number $d\in \N_0$, the following statements are equivalent:
\begin{enumerate}[(a)]
\item $A\leq_{d} B$.
\item There exists a compact self-adjoint operator $F:\ X\to X$ with 
\[
\dim( \range(F))\leq d,\quad \text{ and } \quad 
\langle (B-A +F) x,x\rangle \geq 0 \quad \text{ for all } x\in X.
\]
\item There exists a subspace $W\subset X$ with $\mathrm{codim}(W)\leq d$ such that
\[
\langle (B-A)w,w\rangle \geq 0 \quad \text{ for all } w\in W.
\] 
\item There exists a subspace $V\subset X$ with $\dim (V)\leq d$ such that
\[
\langle (B-A)v,v\rangle \geq 0 \quad \text{ for all } v\in V^\perp.
\]
\end{enumerate}
\end{corollary}
\begin{proof}
This follows from lemma \ref{lemma_eigenvalues_quadform}(a) with $r=0$ and $A$ replaced by $A-B$.
\end{proof}


\begin{lemma}\label{lemma:order_transitive}
Let $A,B,C:\ X\to X$ be self-adjoint compact linear operators on a Hilbert space $X$.
For $d_1,d_2\in \N_0$
\[
A\leq_{d_1} B \quad \text{ and } \quad B\leq_{d_2} C \quad \text{ implies } \quad A\leq_{d_1+d_2} C,
\]
and
\[
A\leq_\text{fin} B \quad \text{ and } \quad B\leq_\text{fin} C \quad \text{ implies } \quad A\leq_\text{fin} C.
\]
\end{lemma}
\begin{proof}
This follows from the characterization in corollary \ref{cor_order_characterization}(b). 
\end{proof}

\subsection{A monotonicity relation for the Helmholtz equation}\label{subsect:monotonicity}

With this new ordering, we can show a monotonicity relation between the
scattering index and the Neumann-to-Dirichlet-operators. Note that the dimension bound in the last line of the following theorem has recently been improved to $d(q_2)-d(q_1)$ in \cite{harrach2019dimension}.
\begin{theorem}\label{thm:monotonicity}
Let $q_1,q_2\in L^\infty(\Omega) \setminus \{0\}$. 
Assume that $k>0$ is not a resonance for $q_1$ or $q_2$, and let $d(q_2)\in \N_0$ be defined 
as in lemma \ref{lemma:d_q}.

Then there exists a subspace $V\subset L^2(\Sigma)$ with $\dim(V)\leq d(q_2)$ such that
\[
\int_{\Sigma} g \left( \Lambda(q_2) - \Lambda(q_1)\right) g \dx[s] \geq
\int_\Omega  k^2 (q_2-q_1) |u^{(g)}_1|^2 \dx
\quad \text{ for all } g\in V^\perp.
\]
In particular
\[
q_1\leq q_2 \quad \text{ implies } \quad \Lambda(q_1)\leq_{d(q_2)} \Lambda(q_2).
\]
\end{theorem}

\begin{remark}\label{remark:othermonotonicity}
Note that by interchanging $q_1$ and $q_2$, theorem \ref{thm:monotonicity} also
yields that there exists a subspace $V\subset L^2(\Sigma)$ with $\dim(V)\leq
d(q_1)$ such that
\[
\int_{\Sigma} g \left( \Lambda(q_2) - \Lambda(q_1)\right) g \dx[s] \leq
\int_\Omega  k^2 (q_2-q_1) |u^{(g)}_2|^2 \dx \quad \text{ for all } g\in V^\perp.
\]
\end{remark}

To prove theorem \ref{thm:monotonicity} we will use the following lemmas.

\begin{lemma}\label{lemma:monotonicity_with_rhs}
Let $q_1,q_2\in L^\infty(\Omega) \setminus \{0\}$. 
Assume that $k>0$ is not a resonance for $q_1$ or $q_2$.
Then, for all $g\in L^2(\Sigma)$,
\begin{align*}
\lefteqn{\int_{\Sigma} g \left( \Lambda(q_2) - \Lambda(q_1)\right) g \dx[s] +
\int_\Omega  k^2 (q_1-q_2) |u^{(g)}_1|^2 \dx}\\
&=\int_\Omega \left( \left| \nabla  (u^{(g)}_2-u^{(g)}_1) \right|^2 - k^2 q_2 |u^{(g)}_2-u^{(g)}_1|^2 \right) \dx.
\end{align*}
where $u^{(g)}_1$, resp., $u^{(g)}_2$ is the solution of the Helmholtz equation
\req{Helmholtz} with Neumann boundary data $g$ and 
$q=q_1$, resp., $q=q_2$.
\end{lemma}
\begin{proof}
Define the bilinear form 
\[
B_q(u, v) = \int_{\Omega} \left( \nabla u \cdot \nabla v - k^2 q u v \right) \dx, \qquad u, v \in H^1(\Omega).
\]
Writing $u_1 = u_1^{(g)}$ and $u_2 = u_2^{(g)}$, from the definition of the NtD map and from \req{Varform} we have 
\begin{align*}
\int_{\Sigma} g \Lambda(q_1) g \dx[s] &= \int_{\Sigma} (\partial_{\nu} u_1) u_1 \dx[s] = 2 \int_{\Sigma}(\partial_{\nu} u_2) u_1 \dx[s] - \int_{\Sigma} (\partial_{\nu} u_1) u_1 \dx[s] \\
 &= 2 B_{q_2}(u_2,u_1) - B_{q_1}(u_1,u_1)
\end{align*}
and 
\begin{align*}
\int_{\Sigma} g \Lambda(q_2) g \dx[s] &= \int_{\Sigma} (\partial_{\nu} u_2) u_2 \dx[s] = B_{q_2}(u_2,u_2).
\end{align*}

We thus obtain that 
\begin{align*}
\int_{\Sigma} g \left( \Lambda(q_2) - \Lambda(q_1)\right) g \dx[s] &= B_{q_2}(u_2,u_2) - 2 B_{q_2}(u_2,u_1) + B_{q_1}(u_1,u_1) \\
 &= B_{q_2}(u_2-u_1, u_2-u_1) - B_{q_2}(u_1,u_1) + B_{q_1}(u_1, u_1).
\end{align*}
This shows the assertion.
\end{proof}

We will show that the bilinear forms in the right hand sides in lemma~\ref{lemma:monotonicity_with_rhs} 
are positive up to a finite dimensional subspace.

\begin{lemma}\label{lemma:monotonicity_hilf2}
Let $q_1,q_2\in L^\infty(\Omega)\setminus \{0\}$ for which $k>0$ is not a resonance.
There exists a subspace $V\subset L^2(\Sigma)$ with $\dim(V)\leq d(q_2)$ such that 
for all $g\in V^\perp$
\[
\int_\Omega \left( \left| \nabla  (u^{(g)}_2-u^{(g)}_1) \right|^2 - k^2 q_2 |u^{(g)}_2-u^{(g)}_1|^2 \right) \dx\geq 0.
\]
\end{lemma}
\begin{proof}
Using lemma \ref{lemma:d_q}, we have that 
\begin{align*}
\int_\Omega \left( \left| \nabla  (u^{(g)}_2-u^{(g)}_1) \right|^2 - k^2 q_2 |u^{(g)}_2-u^{(g)}_1|^2 \right)\dx
 \geq 0
\end{align*}
for all $g\in L^2(\Sigma)$ with $u^{(g)}_2-u^{(g)}_1\in V(q_2)^\perp$.
The solution operators 
\[
S_j:\ L^2(\Sigma)\to H^1(\Omega),\quad g\mapsto u_j^{(g)}, \quad \text{ where } u_j^{(g)}\in H^1(\Omega) \text{ solves \req{Helmholtz},}
\quad j\in \{1,2\},
\]
are linear and bounded, and 
\[
(S_2-S_1)g=u^{(g)}_2-u^{(g)}_1\in V(q_2)^\perp \quad \text{ if and only if } \quad g\in \left( (S_2-S_1)^* V(q_2)\right)^\perp.
\]
Since $\dim(S_2-S_1)^* V(q_2)\leq \dim V(q_2)=d(q_2)$, the assertion follows with 
$V:=(S_2-S_1)^* V(q_2)$.
\end{proof}

\emph{Proof of theorem \ref{thm:monotonicity}.}\
The assertion of theorem~\ref{thm:monotonicity} now immediately follows from combining lemma~\ref{lemma:monotonicity_with_rhs} and lemma~\ref{lemma:monotonicity_hilf2}.
\hfill $\Box$

\subsection{The number of negative eigenvalues}

We will now further investigate the number $d(q)\in \N_0$ (defined in lemma \ref{lemma:d_q}) 
that bounds the number of negative eigenvalues in the monotonicity relations derived in subsection \ref{subsect:monotonicity}. 
We will show that $d(q)$ depends monotonously on the scattering index $q$, and show that $d(q)$ is less or equal than
the number of Neumann eigenvalues for the Laplacian which are larger than $-k^2q_\text{max}$, where $q_\text{max}\geq q(x)$ for all $x\in \Omega$ (a.e.)

\begin{lemma}\label{lemma_monotonicity_q}
Let $q_1,q_2\in L^\infty(\Omega)$, then $q_1\leq q_2$ implies $d(q_1)\leq d(q_2)$.
\end{lemma}
\begin{proof}
$q_1\leq q_2$ implies that $K_{q_1}\leq K_{q_2}$. Hence, the assertion follows
from the equivalence of (a) and (c) in
corollary ~\ref{cor_order_characterization}.
\end{proof}


\begin{lemma}\label{lemma:Neumann_EV}
Let $q \in L^{\infty}(\Omega)$, and $k\in \R$.
\begin{enumerate}[(a)]
\item There is a countable and discrete set of real values
\[
\lambda_1 \geq \lambda_2 \geq \lambda_3 \ldots \to -\infty,
\]
(called \emph{Neumann eigenvalues}) so that
\begin{equation} \labeq{neumann_eigenvalue}
(\Delta + k^2q)u = \lambda u \text{ in $\Omega$}, \qquad \partial_{\nu} u|_{\partial \Omega} = 0,
\end{equation}
admits a non-trivial solution (called \emph{Neumann eigenfunction}) $0\not\equiv u\in H^1(\Omega)$ if and only if $\lambda\in \{ \lambda_1,\lambda_2,\ldots \}$, and there is an orthonormal basis 
$(u_1,u_2,\ldots)$ of $L^2(\Omega)$, so that $u_j\in H^1(\Omega)$ is a Neumann eigenfunction for $\lambda_j$.
\item If $\lambda$ is not a Neumann eigenvalue, then the problem 
\begin{equation} \labeq{neumann_problem_eigenvalue_f_g}
(\Delta + k^2q)u = \lambda u + F \text{ in $\Omega$}, \qquad \partial_{\nu} u|_{\partial \Omega} = g,
\end{equation}
has a unique solution $u \in H^1(\Omega)$ for any $F \in L^2(\Omega)$ and $g
\in L^2(\partial \Omega)$, and the solution operator is linear and bounded. 
\item Let $N_+ := \spn \{ u_j:\ \lambda_j > 0 \}$. Then $\dim(N_+)<\infty$, 
\begin{equation}\labeq{Neumann_EV_Def_Nminus}
N_-:=\overline{\spn \{ u_j:\ \lambda_j \leq 0 \}}=\{v\in H^1(\Omega):\ v\perp_{L^2} N_+\}
\end{equation}
is a complement of $N_+$ (in $H^1(\Omega)$), and
\begin{align}
\labeq{Neumann_EV_N_plus} \int_{\Omega} |\nabla v|^2 - k^2 q v^2 \dx[x] < 0 \quad \text{ for all } v\in N_+,\\
\labeq{Neumann_EV_N_minus} \int_{\Omega} |\nabla v|^2 - k^2 q v^2 \dx[x] \geq 0 \quad \text{ for all } v\in N_-,
\end{align}
where the closure in \req{Neumann_EV_Def_Nminus} is taken with respect to the $H^1(\Omega)$-norm, 
and $\perp_{L^2}$ denotes orthogonality with respect to the $L^2$ inner product.
\item $d(q)$ is the number of positive Neumann eigenvalues of $\Delta+k^2q$, i.e., $d(q)=\dim(N_+)$.
\item $0$ is a Neumann eigenvalue if and only if $k>0$ is a resonance frequency.
\end{enumerate}
\end{lemma}
\begin{proof}
\begin{enumerate}[(a)]
\item Define $c:=k^2\norm{q}_{L^\infty(\Omega)}+1>0$, and  $R:=I-K-k^2 K_q+cK$. Then 
$R$ is coercive and thus continuously invertible. 
Using the equivalent variational formulation of \req{neumann_eigenvalue}, we have that
$\lambda\in \R$ is a Neumann eigenvalue with Neumann eigenfunction $u\not\equiv 0$ if and only if
\[
\int_\Omega \left(-\nabla u\cdot \nabla v + k^2 quv\right)\dx = \lambda \int_\Omega uv\dx
\quad \text{ for all } v\in H^1(\Omega),
\]
which is equivalent to
\[
(I-K-k^2K_q)u=-\lambda K u
\]
and thus to
\begin{equation}\labeq{lemma_Neumann_EV_R_equiv}
Ru=(I-K-k^2 K_q+cK)u=(c-\lambda) K u.
\end{equation}
This shows that $c$ cannot be a Neumann eigenvalue since $Ru\not\equiv 0$ for $u\not\equiv 0$. Moreover, using $K=j^*j$, the invertibility of
$R$, and the injectivity of $j$, we have that \req{lemma_Neumann_EV_R_equiv} is equivalent to 
\[
\frac{1}{c-\lambda} (j u) = j R^{-1} j^* (ju). 
\]
This shows that $\lambda\in \R$ is a Neumann eigenvalue with Neumann eigenfunction $u\in H^1(\Omega)$
if and only if $ju\in L^2(\Omega)$ is an eigenfunction of $j R^{-1} j^*:\ L^2(\Omega)\to L^2(\Omega)$
with eigenvalue $\frac{1}{c-\lambda}$. Since $j$ is injective, and every eigenfunction of $j R^{-1} j^*$ lies in the range of $j$, this is a one-to-one correspondence, and the dimension of the corresponding eigenspaces is the same.
Since $j R^{-1} j^*$ is a compact, self-adjoint, positive operator, the assertions in (a) follow from the spectral theorem on
self-adjoint compact operators.
\item follows from the fact that $I-K-k^2 K_q-\lambda K$ is Fredholm of index $0$ and thus continuously invertible if it is injective.
\item $\dim(N_+)<\infty$ follows from (a).
We define
\[
N_-:=\overline{\spn \{ u_j:\ \lambda_j \leq 0 \}}, \quad \text{ and } \quad \tilde N_-:=\{v\in H^1(\Omega):\ v\perp_{L^2} N_+\}.
\]
$\tilde N_-$ is closed with respect to the $H^1$-norm and contains all $u_j$ with $\lambda_j\leq 0$, so that $N_-\subseteq \tilde N_-$. To
show $N_-=\tilde N_-$, we argue by contradiction. If $N_-\subsetneq \tilde N_-$, then there would exist a $0\neq v\in \tilde N_-$ with
$\langle u_j, v\rangle = 0$ for all $u_j$ with $\lambda_j\leq 0$. 
Using 
\begin{align*}
0&=\langle u_j, v\rangle=\int_\Omega \left( \nabla u_j\cdot \nabla v + u_j v\right) \dx \\
&=\int_\Omega \left( \nabla u_j\cdot \nabla v - k^2 q u_j v\right) \dx
+ \int_\Omega \left( 1+k^2q\right) u_j v \dx\\
&= \int_\Omega \left( 1+k^2q-\lambda_j\right) u_j v \dx,
\end{align*}
and the fact that $\lambda_j\to -\infty$, it would follow that $v\perp_{L^2} u_j$ for all but finitely many $u_j$.
Since $v\perp_{L^2} N_+$, and $(u_1,u_2,\ldots)$ is an orthonormal basis of $L^2(\Omega)$, $v$ must then be a finite combination 
of $u_j$ with $\lambda_j\leq 0$, which would imply that $v=0$. Hence, $N_-=\tilde N_-$, so that the equality in \req{Neumann_EV_Def_Nminus}
is proven. 

Obviously, $N_+\cap N_-=0$ and every $v\in H^1(\Omega)$ can be written as 
\[
v= \sum_{\lambda_j>0} \left( \int_\Omega v u_j\dx \right) u_j + \left( v - \sum_{\lambda_j>0} \left( \int_\Omega v u_j\dx \right) u_j\right) \in N_+ + N_-,
\]
which shows that $N_-$ is a complement of $N_+$.

To show \req{Neumann_EV_N_plus}, we use the $L^2$-orthogonality of the $u_j$ to obtain for all $v=\sum_{\lambda_j > 0} \alpha_j u_j\in N_+$ 
\begin{align*}
\int_\Omega \left( |\nabla v|^2 - k^2qv v\right)\dx 
&= 
\sum_{\lambda_j > 0} \alpha_j \int_\Omega \left( \nabla u_j \cdot \nabla v - k^2qu_j v\right)\dx\\
 & = - \sum_{\lambda_j > 0} \alpha_j \lambda_j \int_{\Omega} u_j v\dx 
 = - \sum_{\lambda_j > 0} \alpha_j^2 \lambda_j \int_{\Omega} u_j^2\dx <0.
\end{align*}
Since every $v\in N_-$ is a 
$H^1(\Omega)$-limit of finite linear combinations of $u_j$ with $\lambda_j\leq 0$,
\req{Neumann_EV_N_minus} follows with the same argument.
\item \req{Neumann_EV_N_plus} can be written as
\[
\langle (K+k^2K_q)v,v\rangle > \norm{v}^2 \quad \text{ for all } v\in N_+.
\]
Lemma \ref{lemma_eigenvalues_quadform}(b) implies that the number $d(q)$ of eigenvalues of  
$K+k^2K_q$ larger than $1$ must be at least $\dim(N_+)$. Likewise, \req{Neumann_EV_N_minus} can be written as
\[
\langle (K+k^2K_q)v,v\rangle \leq \norm{v}^2 \quad \text{ for all } v\in N_-.
\]
Hence, lemma \ref{lemma_eigenvalues_quadform}(a) shows that $d(q)$ is at most 
$\mathrm{codim}(N_-)=\dim(N_+)$.
\item is trivial.
\end{enumerate}
\end{proof}

\begin{corollary}\label{corollary:d_vs_Neumann_EV_Laplacian}
If $q\in L^\infty(\Omega)$ and $q(x)\leq q_\text{max}\in \R$ for all $x\in \Omega$ (a.e.), then 
$d(q)\leq d(q_\text{max})$, and $d(q_\text{max})$ is the number of Neumann eigenvalues of the Laplacian $\Delta$ that are larger than $-k^2q_\text{max}$.
\end{corollary}
\begin{proof}
Obviously, the number of positive Neumann eigenvalues of $\Delta+k^2q_\text{max}$ agrees with 
the number of Neumann eigenvalues of the Laplacian $\Delta$ that are greater than $-k^2q_\text{max}$. Hence, the assertion follows
from lemma \ref{lemma_monotonicity_q} and lemma \ref{lemma:Neumann_EV}(d).
\end{proof}

\begin{remark}
One can show, by using constant potentials, that for the Helmholtz equation, $\Lambda_{q_2} - \Lambda_{q_1}$ can actually have negative eigenvalues when $q_1 \leq q_2$. This shows that in Theorem \ref{thm:monotonicity} it is indeed necessary to work modulo a finite dimensional subspace. The details will appear in a subsequent work.
\end{remark}

\section{Localized potentials for the Helmholtz equation}\label{sect:loc_pot}

In this section we extend the result in \cite{gebauer2008localized} to the Helmholtz equation and prove that we can control the energy terms appearing in the monotonicity relation in spaces of finite codimension. We will first state the result and prove it using a functional analytic relation between operator norms and the ranges of their adjoints in subsection \ref{subsect:localized_potentials_range}. Subsection~\ref{subsect:localized_potentials_runge} then gives an alternative proof that is based on a Runge approximation argument.

\subsection{Localized potentials}\label{subsect:localized_potentials_range}

Our main result on controlling the solutions of the Helmholtz equation in spaces of finite codimension is the following theorem.

\begin{theorem}\label{thm:localized_potentials}
Let $q\in L^\infty(\Omega) \setminus \{0\}$, for which $k>0$ is not a resonance.
Let $B,D\subseteq \overline\Omega$ be measurable, $B\setminus \overline{D}$ possess positive measure, and $\overline{\Omega}\setminus \overline{D}$ be connected to $\Sigma$.

Then for any subspace $V\subset L^2(\Sigma)$ with $\dim V<\infty$, there exists a sequence
$(g_j)_{j\in \N}\subset V^\perp$ such that
\[
\int_{B} |u_{q}^{(g_j)}|^2 \dx \to \infty, \quad \text{ and } \quad \int_{D} |u_{q}^{(g_j)}|^2 \dx \to 0,  
\]
where $u_{q}^{(g_j)}\in H^1(\Omega)$ solves the Helmholtz equation \req{Helmholtz} with Neumann boundary data $g_j$.
\end{theorem}

The arguments that we will use to prove theorem \ref{thm:localized_potentials} in this
subsection also yield a simple proof for the following elementary result. We
formulate it as a theorem since we will utilize it in the next section to
control energy terms in monotonicity inequalities for different scattering
coefficients.

\begin{theorem}\label{thm:useful}
Let $q_1,q_2\in L^\infty(\Omega) \setminus \{0\}$, for which $k>0$ is not a resonance.
If $q_1(x)=q_2(x)$ for all $x$ (a.e.) outside a measurable set $D\subset \Omega$, then there exist constants $c_1,c_2>0$ such
that 
\[
c_1 \int_{D} |u_{1}^{(g)}|^2 \dx \leq \int_{D} |u_{2}^{(g)}|^2 \dx \leq c_2 \int_{D} |u_{1}^{(g)}|^2\dx
\quad \text{ for all } \quad g\in L^2(\Sigma),
\] 
where $u_{1}^{(g)}, u_{2}^{(g)}\in H^1(\Omega)$ solve the Helmholtz equation \req{Helmholtz} 
with Neumann boundary data $g$ and $q=q_1$, resp., $q=q_2$.
\end{theorem}

To prove theorems \ref{thm:localized_potentials} and \ref{thm:useful} we will formulate and prove
several lemmas. Let us first note that the assertion of theorem \ref{thm:localized_potentials} already holds if we
can prove it for a subset of $B$ with positive measure. We will use the subset $B\cap C$, where $C$
is a small closed ball constructed in the next lemma.

\begin{lemma}\label{lemma:connected}
Let $B,D\subseteq \overline\Omega$ be measurable, $B\setminus \overline{D}$ possess positive measure, and $\overline{\Omega}\setminus \overline{D}$ be connected to $\Sigma$. Then there exists a closed ball $C$ such that $B\cap C$ has positive measure, $C\cap \overline{D}=\emptyset$, and $\overline{\Omega}\setminus (\overline{D}\cup C)$ is connected to $\Sigma$.  
\end{lemma}
\begin{proof}
Let $x$ be a point of Lebesgue density one in $B\setminus \overline{D}$. Then the closure $C$ of a sufficiently small ball centered in $x$ will fulfill that $B\cap C$ has positive measure, $C\cap \overline{D}=\emptyset$, and that $\overline{\Omega}\setminus (\overline{D}\cup C)$ is connected to $\Sigma$.
\end{proof}

Now we follow the general approach in \cite{gebauer2008localized}. We formulate the energy terms in theorem \ref{thm:localized_potentials} as norms of operator evaluations and characterize their adjoints. Then we characterize the ranges of the adjoints using the unique continuation property, and then prove theorem \ref{thm:localized_potentials} using a functional analytic relation between norms of operator evaluations and ranges of their adjoints.

\begin{lemma}\label{lemma:adjoint_L}
Let $q\in L^\infty(\Omega) \setminus \{0\}$, for which $k>0$ is not a resonance.
For a measurable set $D\subset \Omega$ we define
\begin{align*}
L_D:\ L^2(\Sigma)\to L^2(D), \quad g\mapsto u|_{D},
\end{align*}
where $u\in H^1(\Omega)$ solves \req{Helmholtz}. Then $L_D$ is a compact linear operator, and its adjoint fulfills
\begin{align*}
L_D^*:\ L^2(D)\to L^2(\Sigma), \quad f\mapsto v|_{\Sigma},
\end{align*} 
where $v$ solves 
\begin{equation}\labeq{PDE_for_adjoint_L}
\Delta v + k^2q v= f \chi_D, \quad \partial_\nu v|_{\partial\Omega}=0. 
\end{equation}
\end{lemma}
\begin{proof}
With the operators $I$, $j$, and $K_{q}$ defined as in subsection
\ref{Subsec:NtD} and \req{Helmholtz_Operatorform} we have that
\begin{equation*}
L_D=R_D j (I-K-k^2 K_q)^{-1}\gamma_\Sigma^*, 
\end{equation*}
where $R_D:\ L^2(\Omega)\to L^2(D)$ is the restriction operator $v\to v|_D$.
Hence, $L_D$ is a linear compact operator, and its adjoint
is
\[
L_D^* = \gamma_\Sigma  (I-K-k^2 K_q)^{-1}j^* R_D^*. 
\]
Thus $L_D^* f=v|_{\Sigma}$ where $v\in H^1(\Omega)$ solves $(I-K-k^2 K_q) v = j^* R_D^* f$, i.e.,
for all $w\in H^1(\Omega)$,
\begin{align*}
\int_{\Omega} \left( \nabla v \cdot \nabla w - k^2 q vw\right) \dx
&= \langle  (I-K-k^2 K_q) v,w \rangle 
=  \langle   j^* R_D^* f , w \rangle
= \int_D f w \dx,
\end{align*}
which is the variational formulation equivalent to \req{PDE_for_adjoint_L}.
\end{proof}

\begin{lemma}\label{lemma:rangeLstar}
Let $q\in L^\infty(\Omega) \setminus \{0\}$, for which $k>0$ is not a resonance.
Let $B,D\subseteq \overline\Omega$ be measurable, and $C\subseteq \overline{\Omega}$ be a closed set such that $B\cap C$ has positive measure, $C\cap \overline{D}=\emptyset$, and $\overline{\Omega}\setminus (\overline{D}\cup C)$ is connected to $\Sigma$. 
Then, 
\begin{equation}\labeq{range_intersection_trivial}
\range(L_{B\cap C}^*)\cap \range(L_D^*)=\{0\}.
\end{equation}
and $\range(L_{B\cap C}^*), \range(L_D^*)\subset L^2(\Sigma)$ are both dense (and thus in particular infinite dimensional).
\end{lemma}
\begin{proof}
It follows from the unique continuation property in theorem~\ref{thm:UCP}(a) that $L_{B\cap C}$ and $L_D$ are injective. Hence $\range(L_{B\cap C}^*)$ and $\range(L_D^*)$ are dense subspaces of $L^2(\Sigma)$. 

The characterization of the adjoint operators in lemma~\ref{lemma:adjoint_L} shows that 
\[
B\cap C\subseteq C \quad \text{ implies that } \quad \range(L_{B\cap C}^*)\subseteq \range(L_C^*).
\]
Hence, \req{range_intersection_trivial} follows a fortiori if we can show that
\[
\range(L_C^*)\cap \range(L_D^*)=\{0\}.
\]
To show this let $h\in \range(L_C^*)\cap \range(L_D^*)$. Then there exist
$f_C\in L^2(C)$, $f_D\in L^2(D)$, and $v_C,v_D\in H^1(\Omega)$ such that
\begin{align*}
\Delta v_C + k^2q v_C= f_C \chi_C, \quad \partial_\nu v|_{\partial \Omega}=0,\\
\Delta v_D + k^2q v_D= f_D \chi_D, \quad \partial_\nu v|_{\partial \Omega}=0, 
\end{align*}
and $v_C|_{\Sigma}=h=v_D|_{\Sigma}$. 

It follows from the unique continuation property in theorem~\ref{thm:UCP}(b) that $v_C=v_D$ on the connected set $\Omega\setminus (C\cup \overline{D})$. Hence, 
\[
v:=\left\{ \begin{array}{l l} v_C=v_D & \text{ on $\Omega\setminus (C\cup \overline{D})$}\\
v_C & \text{ on $\overline{D}$}\\
v_D & \text{ on $C$}
\end{array} \right.
\]
defines a $H^1(\Omega)$-function solving 
\[
\Delta v + k^2 q v= 0, \quad \partial_\nu v|_{\partial \Omega}=0,\\
\]  
so that $v=0$ and thus $h=v_C|_{\Sigma}=v_D|_{\Sigma}=v|_{\Sigma}=0$.
\end{proof}

\begin{lemma}\label{lemma:fundamental}
Let $X$, $Y$ and $Z$ be Hilbert spaces, and $A_1:\ X\to Y$ and $A_2:\ X\to Z$ be linear bounded operators.
Then
\[
\exists c>0:\ \norm{A_1x}\leq c \norm{A_2 x} \quad \forall x\in X
\quad \text{ if and only if } \quad \range(A_1^*)\subseteq \range(A_2^*). 
\] 
\end{lemma}
\begin{proof}
This is proven for reflexive Banach spaces in \cite[Lemma~2.5]{gebauer2008localized}. Note that one direction of the implication
also holds in non-reflexive Banach spaces, see \cite[Lemma~2.4]{gebauer2008localized}. 
\end{proof}

\begin{lemma}\label{lemma:vector_spaces}
Let $V,X,Y\subset Z$ be subspaces of a real vector space $Z$. If 
\[
X\cap Y= \{0\}, \quad \text{ and } \quad X\subseteq Y+ V,
\]
then $\dim(X)\leq \dim(V)$.
\end{lemma}
\begin{proof}
Let $(x_j)_{j=1}^m\subset X$ be a linearly independent sequence of $m$ vectors. Then there exist
$(y_j)_{j=1}^m\subset Y$ and $(v_j)_{j=1}^m\subset V$ such that $x_j=y_j+v_j$ for all $j=1,\ldots,m$.
We will prove the assertion by showing that the sequence $(v_j)_{j=1}^m$ is linearly independent.
To this end let $\sum_{j=1}^m a_j v_j=0$ with $a_j\in \R$, $j=1,\ldots,m$. Then
\[
\sum_{j=1}^m a_j x_j = \sum_{j=1}^m a_j (y_j +v_j)= \sum_{j=1}^m a_j y_j\in Y,
\] 
so that $\sum_{j=1}^m a_j x_j=0$. Since  $(x_j)_{j=1}^m\subset X$ is linearly independent, it follows that 
$a_j=0$ for all $j=1,\ldots,m$. This shows that $(v_j)_{j=1}^m$ is linearly independent.
\end{proof}

\emph{Proof of theorem \ref{thm:localized_potentials}.}
Let $q\in L^\infty(\Omega) \setminus \{0\}$, for which $k>0$ is not a resonance.
Let $B,D\subseteq \overline\Omega$ be measurable, $B\setminus \overline{D}$ possess positive measure, and $\overline{\Omega}\setminus \overline{D}$ be connected to $\Sigma$. Using lemma \ref{lemma:connected} we obtain a 
closed set $C\subseteq \overline{\Omega}$ such that $B\cap C$ has positive measure, $C\cap \overline{D}=\emptyset$, and $\overline{\Omega}\setminus (\overline{D}\cup C)$ is connected to $\Sigma$.

Let $V\subset L^2(\Sigma)$ be a subspace with $d:=\dim(V)<\infty$. Since $V$ is finite dimensional and thus closed, there exists an orthogonal projection operator 
$P_V:\ L^2(\Sigma)\to L^2(\Sigma)$ with
\[
\range(P_V)=V, \quad P_V^2=P_V, \quad \text{ and } \quad P_V=P_V^*.
\]

From lemma~\ref{lemma:rangeLstar}, we have that $\range(L_{B\cap C}^*)\cap \range(L_D^*)=0$ and that $\range(L_{B\cap C}^*)$ is infinite dimensional.
So it follows from lemma~\ref{lemma:vector_spaces} that 
\[
\range(L_{B\cap C}^*)\not\subseteq \range(L_D^*)+ V=\range(L_D^*)+ \range(P_V^*).
\]
Since $B\cap C\subseteq B$ implies that $\range(L_{B\cap C}^*)\subseteq \range(L_{B}^*)$,
and since (using block operator matrix notation) 
\[
\range\left( \begin{pmatrix} L_D^* & P_V^*\end{pmatrix} \right)\subseteq \range(L_D^*)+ \range(P_V^*),
\]
we obtain that 
\[
\range(L_B^*)\not\subseteq \range\left( \begin{pmatrix} L_D^* & P_V^*\end{pmatrix} \right) 
=\range\left( \begin{pmatrix} L_D\\ P_V \end{pmatrix}^* \right).
\]
It then follows from lemma~\ref{lemma:fundamental} that there cannot exist a constant $C>0$ with
\[
\norm{L_B g}^2\leq C^2 \norm{\begin{pmatrix} L_D\\ P_V \end{pmatrix} g}^2
=C^2 \norm{L_D g}^2 + C^2 \norm{P_V g}^2 \quad \forall g\in L^2(\Sigma). 
\] 
Hence, there must exist a sequence $(\tilde g_k)_{k\in \N}\subseteq L^2(\Sigma)$ with
\[
\norm{L_B \tilde g_k}\to \infty, \quad \text{ and } \quad \norm{L_D \tilde g_k},\norm{P_V \tilde g_k}\to 0. 
\]
Thus, $g_k:=\tilde g_k-P_V \tilde g_k\in V^\perp\subseteq L^2(\Sigma)$ and 
\[
\norm{L_B g_k}\geq \norm{L_B \tilde g_k} - \norm{L_B} \norm{P_V \tilde g_k} \to \infty, 
\quad \text{ and } \quad \norm{L_D g_k}\to 0, 
\]
which shows the assertion. \hfill $\Box$

\emph{Proof of theorem \ref{thm:useful}.}
Let $q_1,q_2\in L^\infty(\Omega)$, for which $k>0$ is not a resonance, and let
$q_1(x)=q_2(x)$ for all $x$ (a.e.) outside a measurable set $D\subset \Omega$.
We denote by $L_{q_1,D}$ and $L_{q_2,D}$ the operators from lemma \ref{lemma:adjoint_L}
for $q=q_1$ and $q=q_2$. For $f\in L^2(D)$, we then have 
\[
L_{q_1,D}^* f=v_1|_{\Sigma}\quad \text{ and } \quad L_{q_2,D}^* f =v_2|_{\Sigma}
\]
where $v_1,v_2\in H^1(\Omega)$ solve
\begin{align*}
\Delta v_1 + k^2q_1 v_1&= f \chi_D, \quad \partial_\nu v_1|_{\partial \Omega}=0,\\
\Delta v_2 + k^2q_2 v_2&= f \chi_D, \quad \partial_\nu v_2|_{\partial \Omega}=0. 
\end{align*}
Since this also implies that
\begin{align*}
\Delta v_1 + k^2q_2 v_1&= f\chi_D + k^2 (q_2-q_1) v_1, \quad \partial_\nu v_1|_{\partial \Omega}=0,\\
\Delta v_2 + k^2q_1 v_2&= f\chi_D + k^2 (q_1-q_2) v_2, \quad \partial_\nu v_2|_{\partial \Omega}=0, 
\end{align*}
and $q_1-q_2$ vanishes (a.e.) outside $D$, it follows that
\[
v_1|_{\Sigma}=L_{q_2,D}^* (f+ k^2 (q_2-q_1) v_1) \quad \text{ and } \quad 
v_2|_{\Sigma}=L_{q_1,D}^* (f+ k^2 (q_1-q_2) v_2).
\]
Hence, $\range(L_{q_1,D}^*)=\range(L_{q_2,D}^*)$, so that the assertion follows from
lemma~\ref{lemma:fundamental}.
\hfill $\Box$.

\subsection{Localized potentials and Runge approximation}\label{subsect:localized_potentials_runge}

In this subsection we give an alternative proof of theorem \ref{thm:localized_potentials} that is based on a Runge
approximation argument that characterizes whether a given function $\varphi\in L^2(O)$ on a measurable subset $O\subseteq \Omega$ can be approximated by functions in a subspace of solutions of the Helmholtz equation in $\Omega$. Throughout this subsection let $q\in L^\infty(\Omega) \setminus \{0\}$, for which $k>0$ is not a resonance. We will prove the following theorem.

\begin{theorem}\label{thm:runge_main}
Let $D\subseteq \Omega$ be a measurable set and $C\subset \Omega$ be a closed ball for which $C\cap \overline{D}=\emptyset$, and $\overline{\Omega}\setminus (C\cup \overline{D})$ is connected to $\Sigma$. 

Then for any subspace $V\subset L^2(\Sigma)$ with $\dim V<\infty$, there exists a function $\varphi\in L^2(C\cup \overline{D})$
that can be approximated (in the $L^2(C\cup \overline{D})$-norm) by solutions $u\in H^1(\Omega)$ of 
\[
(\Delta +k^2 q) u=0\text{ in $\Omega$} \quad \text{ with } \quad \partial_\nu u|_{\partial\Omega\setminus \Sigma}=0, \quad \partial_\nu u|_{\Sigma}\in V^\perp,
\]
and fulfills 
\[
\varphi|_{\overline{D}}\equiv 0, \quad \text{ and } \quad \varphi|_{B}\not\equiv 0,
\]
for all subsets $B\subseteq C$ with positive measure.
\end{theorem}

Before we prove theorem~\ref{thm:runge_main}, let us first show that it implies
theorem \ref{thm:localized_potentials}.

\begin{corollary}
Let $B,D\subseteq \overline\Omega$ be measurable, $B\setminus \overline{D}$ possess positive measure, and $\overline{\Omega}\setminus \overline{D}$ be connected to $\Sigma$. Then for any subspace $V\subset L^2(\Sigma)$ with $\dim V<\infty$, there exists a sequence
$(g_j)_{j\in \N}\subset V^\perp$ such that
\[
\int_{B} |u_{q}^{(g_j)}|^2 \dx \to \infty, \quad \text{ and } \quad \int_{D} |u_{q}^{(g_j)}|^2 \dx \to 0,  
\]
where $u_{q}^{(g_j)}\in H^1(\Omega)$ solves the Helmholtz equation \req{Helmholtz} with Neumann boundary data $g_j$.
\end{corollary}
\begin{proof}
As in lemma~\ref{lemma:connected}, we can find a closed ball $C\subset \Omega$, so that $B\cap C$ has positive measure, $C\cap \overline{D}=\emptyset$, and that $\overline{\Omega}\setminus (\overline{D}\cup C)$ is connected to $\Sigma$.
Using theorem \ref{thm:runge_main}, there exists $\varphi\in L^2(C\cup \overline D)$ and
a sequence of solutions
$(\tilde u^{(j)})_{j\in \N}\subset H^1(\Omega)$ of $(\Delta +k^2 q) \tilde
u^{(j)}=0$ in $\Omega$ with $\partial_\nu \tilde u^{(j)}|_{\partial\Omega\setminus \Sigma}=0$, $\partial_\nu \tilde u^{(j)}|_{\Sigma}\in V^\perp$, 
\[
\norm{\tilde u^{(j)}|_{B\cap C}}_{L^2(B\cap C)}\to \norm{\varphi}_{L^2(B\cap C)}>0, \quad \text{ and } \quad \norm{\tilde u^{(j)}|_{\overline{D}}}_{L^2(\overline{D})}\to 0.
\]
Obviously, the scaled sequence 
\[
g^{(j)}:=\frac{\partial_\nu \tilde u^{(j)}}{\sqrt{\norm{\tilde u^{(j)}|_{\overline{D}}}_{L^2(\overline{D})}}}\in V^\perp
\]
fulfills the assertion.
\end{proof}

To prove theorem~\ref{thm:runge_main}, we start with an abstract characterization showing whether a given function $\varphi\in L^2(O)$ on a measurable set $O\subseteq \Omega$ is a limit of functions from a subspace of solutions of the Helmholtz equation in $\Omega$. For the sake of readability, we write $v\chi_O\in L^2(\Omega)$ for the zero extension of a function $v\in L^2(O)$, and we write the dual pairing on $H^{-1/2}(\partial \Omega)\times H^{1/2}(\partial \Omega)$ as an integral over $\partial \Omega$.

\begin{lemma}\label{lemma:runge}
Let $O\subseteq \Omega$ be measurable. Let $H\subseteq H^1(\Omega)$ be a (not necessarily closed) subspace of solutions of $(\Delta +k^2 q) u=0$ in $\Omega$.

A function $\varphi\in L^2(O)$ can be approximated on $O$ by solutions $u\in H$ in the sense that
\[
\inf_{u\in H} \norm{\varphi-u}_{L^2(O)}=0
\]
if and only if
$
\int_O \varphi v \dx=0 
$
for all $v\in L^2(O)$ for which the solution $w\in H^1(\Omega)$ of
\begin{equation}\labeq{runge:def_w}
(\Delta+k^2q)w=v\chi_O \quad \text{ and } \quad \partial_\nu w|_{\partial \Omega}=0
\end{equation}
fulfills that $\int_{\partial \Omega} \partial_\nu u|_{\partial \Omega} w|_{\partial \Omega} \dx[s]=0  \text{ for all } u\in H$.
\end{lemma}
\begin{proof}
Let 
\[
\mathcal{R} := \{ u|_O \,;\, u \in H \}\subseteq L^2(O).
\]
Let $v\in L^2(O)$ and $w\in H^1(\Omega)$ solve \req{runge:def_w}. Then $v\in \mathcal{R}^\perp$ if and only if, for all $u\in H$,
\begin{align*}
0&=\int_O u v\dx = \int_\Omega u (\Delta+k^2q)w \dx = \int_\Omega w(\Delta+k^2q)u\dx - \int_{\partial \Omega} \partial_\nu u|_{\partial \Omega} w|_{\partial \Omega}\dx[s]\\
&= - \int_{\partial \Omega} \partial_\nu u|_{\partial \Omega} w|_{\partial \Omega}\dx[s].
\end{align*}
Hence, the assertion follows from $\overline{\mathcal{R}}=(\mathcal{R}^\perp)^\perp$ (where orthogonality and closures are taken with respect to the $L^2(O)$-inner product).
\end{proof}

Now we characterize the functions $w$ appearing in lemma~\ref{lemma:runge} for
a setting that will be considered in the proof of theorem~\ref{thm:runge_main}.
\begin{lemma}\label{lemma:runge_finmany_w}
Let $V$ be a finite-dimensional subspace of $L^2(\Sigma)$, and $O\subset \Omega$ be a closed set 
for which the complement $\overline \Omega \setminus O$ is connected to $\Sigma$.

We define the spaces
\begin{align*}
W & :=\{ w\in H^1(\Omega):\ \exists v\in L^2(O)\text{ s.t. } (\Delta+k^2q)w=v\chi_O,\ \partial_\nu w|_{\partial \Omega}=0,\  w|_{\Sigma}\in V \},\\
W_0 & :=\{ w\in H^1(\Omega):\ \exists v\in L^2(O)\text{ s.t. } (\Delta+k^2q)w=v\chi_O,\ \partial_\nu w|_{\partial \Omega}=0,\  w|_{\Sigma}=0 \}.
\end{align*}
Then the codimension $d:=\dim(W/W_0)$ of $W_0$ in $W$ is at most $\dim(V)$, i.e., there 
exists functions $w_1,\ldots,w_d\in W$ such that every $w\in W$ can be written as
\[
w=w_0+\sum_{j=1}^d a_j w_j
\]
with ($w$-dependent) $w_0\in W_0$ and $a_1,\ldots,a_d\in \R$.
\end{lemma}
\begin{proof}
$W_0$ is the kernel of the restricted trace operator
\[
\gamma_{\Sigma}|_W: W\to V,\quad w\mapsto w|_{\Sigma}.
\]
Hence, the codimension of $W_0$ as a subspace of $W$ is
\[
\dim(W / W_0)= \dim(\range(\gamma_{\Sigma}|_W))\leq \dim(V),
\]
which proves the assertion.
\end{proof}

\emph{Proof of theorem~\ref{thm:runge_main}.}
Let $D\subseteq \Omega$ be a measurable set and $C\subset \Omega$ be a closed ball for which $C\cap \overline{D}=\emptyset$, and $\overline{\Omega}\setminus (C\cup \overline{D})$ is connected to $\Sigma$. Let $V$ be a finite-dimensional subspace of $L^2(\Sigma)$.

To apply lemma \ref{lemma:runge}, we set $O:=C\cup \overline{D}$ and 
\[
H:=\left\{ u\in H^1(\Omega):\ (\Delta +k^2 q) u=0\text{ in $\Omega$}, \quad \partial_\nu u|_{\partial\Omega\setminus \Sigma}=0, \quad  \partial_\nu u|_{\Sigma}\in V^\perp \right\}.
\]
Then $w\in H^1(\Omega)$ fulfills \req{runge:def_w} and 
 $\int_{\partial \Omega} \partial_\nu u|_{\partial \Omega} w|_{\partial \Omega} \dx[s]=0  \text{ for all } u\in H$
if and only if $w\in W$, with $W$ defined in lemma~\ref{lemma:runge_finmany_w}. Hence, by lemma \ref{lemma:runge}, 
a function $\varphi\in L^2(C\cup \overline{D})$ can be approximated by solutions $u\in H$ if and only if
\begin{equation}\labeq{runge_main_property_varphi}
\int_{C\cup \overline{D}} \varphi (\Delta +k^2 q) w \dx = 0 \quad \text{ for all } w\in W.
\end{equation}
Thus, the assertion of theorem~\ref{thm:runge_main} follows if we can show that there exists 
$\varphi\in L^2(C\cup \overline{D})$ that fulfills \req{runge_main_property_varphi} and vanishes on $D$ but not on any subset of $C$ having positive measure.

To construct such a $\varphi$, we first note that the Helmholtz equation \req{Helmholtz} on $\Omega$ is uniquely solvable for all Neumann data $g\in L^2(\Sigma)$,
and by unique continuation, linearly independent Neumann data yield solutions whose restrictions to the open ball $C^\circ$ are linearly independent. 
Hence, there exists an infinite number of linearly independent solutions
\begin{align}\labeq{thm_rung_main_varphi_j}
\varphi_j\in H^1(C^\circ) \quad \text{ with } \quad (\Delta + k^2 q)\varphi_j =0 \quad \text{ in $C^\circ$},\quad j\in \N.
\end{align}
We extend $\varphi_j$ by zero on $\overline{D}\cup \partial C$ to $\varphi_j\in L^2(O)$. 

Every $w_0\in W_0$, with $W_0$ from lemma~\ref{lemma:runge_finmany_w}, must possess zero Cauchy data $w_0|_{\partial C}=0$ and $\partial_\nu w_0|_{\partial C}=0$ by unique continuation. Hence, for all $w_0\in W_0$, and $j\in \N$,
\begin{align*}
\int_{O} \varphi_j (\Delta + k^2q) w_0 \dx & = \int_{C} \varphi_j (\Delta + k^2q) w_0 \dx\\
&=\int_{\partial C} \left(\varphi_j|_{\partial C} \partial_\nu w_0|_{\partial C} - \partial_\nu \varphi_j|_{\partial C} w_0|_{\partial C} \right)\dx[s] = 0.
\end{align*}

Moreover, by a dimensionality argument, there must exist a non-trivial finite linear combination $\varphi$ of
the infinitely many linearly independent $\varphi_j$, so that 
\[
\int_{O} \varphi (\Delta + k^2q) w_k \dx = 0
\]
for the finitely many functions $w_1,\ldots,w_d\in W$ from lemma~\ref{lemma:runge_finmany_w}. Thus, using lemma~\ref{lemma:runge_finmany_w}, we have constructed a function
$\varphi\in L^2(O)$ with $\varphi|_{\overline D}\equiv 0$, $\varphi|_{C^\circ}\not\equiv 0$, and
\[
\int_{O} \varphi (\Delta + k^2q) w \dx = 0 \quad \text{ for all } w\in W=W_0+\spn\{w_1,\ldots,w_d\}.
\]
Moreover, $\varphi$ solves \req{thm_rung_main_varphi_j}, so that the unique continuation result from measurable sets in theorem~\ref{thm:UCP} also yields that $\varphi|_{B}\not\equiv 0$ for all $B\subseteq C^\circ$ with positive measure. Since $\partial C$ is a null set, the latter also holds for all
$B\subseteq C$ with positive measure. As explained above, the assertion of theorem~\ref{thm:runge_main} now follows from lemma \ref{lemma:runge}.
$\Box$

\section{Local uniqueness for the Helmholtz equation}
\label{Sec:Local}

We are now able to prove the first main result in this work, announced as theorem \ref{intro:thm:local_uniqueness} in the introduction,
and extend the local uniqueness result in \cite{harrach2017local} to the case of negative potentials, and $n\geq 2$.

As in subsection \ref{Subsec:NtD}, let $\Omega\subset \R^{n}$, $n\geq 2$ denote
a bounded Lipschitz domain, and let 
$\Sigma\subseteq \partial \Omega$ be an arbitrarily small, relatively open part
of the boundary $\partial \Omega$.
For $q_1,q_2\in L^\infty(\Omega)$ let 
\[
\Lambda(q_1),\Lambda(q_2):\ L^2(\Sigma)\to L^2(\Sigma), \quad \Lambda(q_1):\ g\mapsto u_1|_{\Sigma},
 \quad \Lambda(q_2):\ g\mapsto u_2|_{\Sigma}
\]
be the Neumann-to-Dirichlet operators for the Helmholtz equation
\begin{equation}\labeq{Uniqueness:Helmholtz}
(\Delta +k^2 q) u = 0 \quad \text{ in } \Omega, \quad \partial_\nu
u|_{\partial \Omega}=\left\{ \begin{array}{l l} g & \text{ on $\Sigma$,}\\0
& \text{else,}\end{array}\right.
\end{equation}
with $q=q_1$, resp., $q=q_2$, and let $k>0$ be not a resonance, neither for $q_1$ nor $q_2$. 

\begin{theorem}\label{thm:local}
Let $q_1\leq q_2$ in a relatively open set $O\subseteq \overline \Omega$ that is connected to $\Sigma$. Then
\[
q_1|_O\not\equiv  q_2|_O \quad \text{ implies } \quad \Lambda(q_1)\neq \Lambda(q_2).
\] 
Moreover, in that case, $\Lambda(q_2)-\Lambda(q_1)$ has infinitely many positive eigenvalues.
\end{theorem}
\begin{proof}
If $q_1|_O\not\equiv  q_2|_O$ then there exists a subset $B\subseteq O$ with positive measure, and a constant
$c>0$ such that $q_2(x)-q_1(x)\geq c$ for all $x\in B$ (a.e.). 
From the monotonicity inequality in theorem \ref{thm:monotonicity} we have that $\Lambda(q_2) - \Lambda(q_1) \geq_\text{fin} A$, where 
\begin{align*}
A: L^2(\Sigma) \to L^2(\Sigma), \ \ \int_{\Sigma} h Ag \dx[s] = \int_\Omega  k^2 (q_2-q_1) u^{(g)}_1 u^{(h)}_1 \dx.
\end{align*}
Note that $A = S_{1}^* j^* k^2 M_{q_1-q_2} j S_{1}$ where $S_1: g \mapsto u_1^{(g)}$ is the solution operator and $j: H^1(\Omega) \to L^2(\Omega)$ is the compact inclusion, so $A$ is indeed a compact, self-adjoint linear operator on $L^2(\Sigma)$.

We will now prove the assertion by contradiction and assume that $\Lambda(q_2)-\Lambda(q_1)\leq_\text{fin} 0$.
Then, the transitivity result in lemma~\ref{lemma:order_transitive} gives that $A \leq_{\text{fin}} 0$. By the characterization in corollary \ref{cor_order_characterization}, there would exist a finite dimensional subspace $V\subseteq L^2(\partial \Omega)$,
with 
\begin{align*}
0 & \geq  \int_\Omega  k^2 (q_2-q_1) |u^{(g)}_1|^2 \dx
= \int_O  k^2 (q_2-q_1) |u^{(g)}_1|^2 \dx + \int_{\Omega\setminus O}  k^2 (q_2-q_1) |u^{(g)}_1|^2 \dx\\
& \geq c \int_B k^2 |u^{(g)}_1|^2 \dx - C  \int_{\Omega\setminus O}  k^2 |u^{(g)}_1|^2 \dx
\end{align*}
for all $g\in V^\perp$, where $C:=\left( \norm{q_1}_{L^\infty(\Omega)} + \norm{q_2}_{L^\infty(\Omega)} \right)$
and $u^{(g)}_1$ solves \req{Uniqueness:Helmholtz} with $q=q_1$.

However, using the localized potentials from theorem \ref{thm:localized_potentials} with $D:=\overline \Omega\setminus O$, there must exist a Neumann datum $g\in V^\perp$ with 
\[
c \int_B k^2 |u^{(g)}_1|^2 \dx > C \int_{\Omega\setminus O}  k^2 |u^{(g)}_1|^2 \dx,
\]
which contradicts the above inequality. Hence, $\Lambda(q_2)-\Lambda(q_1)$ must have infinitely many negative eigenvalues, and in particular 
$\Lambda(q_2)\neq \Lambda(q_1)$. 
\end{proof}

\emph{Proof of theorem \ref{intro:thm:local_uniqueness}.} The result is an
immediate consequence of theorem \ref{thm:local}. \qquad $\Box$

Theorem~\ref{thm:local} shows that two scattering coefficient functions can be distinguished from knowledge of 
the partial boundary measurements if their difference is of definite sign in a neighborhood of $\Sigma$ (or any open subset of $\Sigma$ since $\Lambda(\Sigma)$ determines the boundary measurements on all smaller parts).
This definite sign condition is fulfilled for piecewise-analytic functions, cf., e.g., \cite[Thm.\ A.1]{harrach2013monotonicity}, but the authors are not aware of other 
named function spaces, with less regularity, where infinite oscillations between positive and negative values when approaching the boundary can be ruled out. In the following corollary the term piecewise-analytic is understood with respect to
a partition in finitely many subdomains with piecewise $C^\infty$-boundaries, cf.\  \cite{harrach2013monotonicity}
for a precise definition.

\begin{corollary}\label{cor:pcw_anal_uniqueness}
If $q_1-q_2$ is piecewise-analytic on $\Omega$ then 
\[
\Lambda(q_1)=\Lambda(q_2)\quad \text{ if and only if } \quad q_1=q_2.
\]
\end{corollary}
\begin{proof}
This follows from theorem \ref{intro:thm:local_uniqueness} and \cite[Thm.\ A.1]{harrach2013monotonicity}.
\end{proof}

\section{Detecting the support of a scatterer}
\label{Sec:Detection}

We will now show that an unknown scatterer, where the refraction index is
either higher or lower than an otherwise homogeneous background value, can be reconstructed by simple monotonicity comparisons.

\subsection{Scatterer detection by monotonicity tests}

As before, let $\Omega\subset \R^{n}$, $n\geq 2$ be a bounded domain with Lipschitz boundary.
The domain is assumed to contain an open set (the scatterer) $D\subseteq \Omega$ with $\overline D\subset \Omega$ and connected complement $\Omega\setminus \overline{D}$. We assume that the scattering index fulfills $q(x)=1$ in $\Omega\setminus D$ (a.e.) and 
that there exist constants $q_\text{min}, q_\text{max}\in \R$ so that either
\[
1<  q_\text{min}\leq  q(x) \leq   q_\text{max} \quad \text{ for all } x\in D \text{ (a.e.)},
\]
or 
\[
q_\text{min}\leq  q(x) \leq  q_\text{max}<1 \quad \text{ for all } x\in D \text{ (a.e.)}.
\]
$\Lambda(q)$ denotes the Neumann-to-Dirichlet operator for the domain
containing the scatterer, and $\Lambda(1)$ is the Neumann-to-Dirichlet operator
for a homogeneous domain with $q\equiv 1$. For both cases, we assume that $k>0$
is not a resonance.

For an open set $B\subseteq \Omega$ (e.g., a small ball), we define the operator
\[
T_B:\ L^2(\Sigma)\to L^2(\Sigma), \quad
\int_{\Sigma} g T_B h\dx[s]:=\int_B k^2 u^{(g)}_1 u^{(h)}_1\dx,
\]
where $u^{(g)}_1,u^{(h)}_1\in H^1(\Omega)$ solve \req{Helmholtz} with $q\equiv 1$ and
Neumann boundary data $g$, resp., $h$.
Obviously, $T_B$ is a compact self-adjoint linear operator. 

The following two theorems show that $D$ can be reconstructed by comparing $\Lambda(q)-\Lambda(1)$ with $T_B$ 
in the sense of the Loewner order up to finitely many eigenvalues introduced in subsection~\ref{subsect:Loewner_order}.

\begin{theorem}\label{thm:main1}
Let 
\[
1<  q_\text{min}\leq  q(x) \leq   q_\text{max} \quad \text{ for all } x\in D \text{ (a.e.)},
\]
and let $d( q_\text{max})$ be defined as in lemma \ref{lemma:d_q} (which also equals the number of Neumann eigenvalues of the 
Laplacian $\Delta$ that are larger than $-k^2 q_\text{max}$, cf.\ corollary \ref{corollary:d_vs_Neumann_EV_Laplacian}).
\begin{enumerate}[(a)]
\item If $B\subseteq D$ then 
\[
\alpha T_B \leq_{d(  q_\text{max})} \Lambda(q)-\Lambda(1)\quad 
\text{ for all } \alpha\leq  q_\text{min}-1.
\]
\item If $B\not \subseteq D$ then, for all $\alpha>0$,
    $\Lambda(q)-\Lambda(1)-\alpha T_B$ has infinitely many negative
eigenvalues.
\end{enumerate}
\end{theorem}

\begin{theorem}\label{thm:main2}
Let 
\[
q_\text{min}\leq  q(x) \leq  q_\text{max}<1 \quad \text{ for all } x\in D \text{ (a.e.)},
\]
and let $d(1)$ be defined as in lemma \ref{lemma:d_q} (which also equals the number of Neumann eigenvalues of the 
Laplacian $\Delta$ that are larger than $-k^2$, cf.\ corollary \ref{corollary:d_vs_Neumann_EV_Laplacian}).
\begin{enumerate}[(a)]
\item If $B\subseteq D$ then there exists $\alpha_\text{max}>0$ such that
\[
\alpha T_B \leq_{d(1)} \Lambda(1)-\Lambda(q)\quad \text{ for all } \alpha\leq \alpha_\text{max}.
\]
\item If $B\not \subseteq D$ then, for all $\alpha>0$, $\Lambda(1)-\Lambda(q) - \alpha T_B$ 
has infinitely many negative eigenvalues.
\end{enumerate}
\end{theorem}

\subsection{Proof of theorem \ref{thm:main1} and \ref{thm:main2}}

We prove both results by combining the monotonicity relations and localized
potentials results from the last subsections.

\emph{Proof of theorem \ref{thm:main1}.}
By the monotonicity relation in theorem \ref{thm:monotonicity} there exists a
subspace $V\subset L^2(\Sigma)$ with
$\dim(V)\leq d( q)\leq d(q_\text{max})$ (cf.\ corollary~\ref{corollary:d_vs_Neumann_EV_Laplacian}) and
\[
\int_{\Sigma} g \left( \Lambda(q) - \Lambda(1)\right) g \dx[s] \geq  \int_\Omega  k^2 (q-1) |u^{(g)}_1|^2 \dx
\quad \text{ for all } g\in V^\perp.
\]
If $B\subseteq D$ and $\alpha\leq  q_\text{min}-1$, then $q-1\geq \alpha\chi_B$, so that for all $g\in L^2(\Sigma)$
\[
\int_\Omega  k^2 (q-1) |u^{(g)}_1|^2 \dx\geq \int_B  k^2 \alpha |u^{(g)}_1|^2 \dx = \alpha \int_{\Sigma} g T_B g.
\]
Hence, if $B\subseteq D$ and $\alpha\leq  q_\text{min}-1$, then 
\[
\int_{\Sigma} g \left( \Lambda(q) - \Lambda(1)\right) g \dx[s] \geq \alpha \int_{\Sigma} g T_B g 
\quad 
\text{ for all } g\in V^\perp
\]
which proves (a).

To prove (b) by contradiction, let $B\not\subseteq D$, $\alpha>0$, and assume that
\begin{equation}\labeq{main1_hilf1}
\Lambda(q)-\Lambda(1)\geq_\text{fin}\alpha T_B.
\end{equation}
Using the monotonicity relation in remark~\ref{remark:othermonotonicity} together with theorem~\ref{thm:useful}, there exists
a finite-dimensional subspace $V\subset L^2(\Sigma)$ and a constant $C>0$, so that for all $g\in V^\perp$
\begin{equation}\labeq{main1_hilf2}
\int_{\Sigma} g \left( \Lambda(q) - \Lambda(1)\right) g \dx[s] \leq
\int_D  k^2 (q-1) |u^{(g)}_q|^2 \dx
\leq C \int_D  k^2 (q-1) |u^{(g)}_1|^2 \dx.
\end{equation}

Combining \req{main1_hilf1} and \req{main1_hilf2} using the transitivity result from
lemma~\ref{lemma:order_transitive}, there exists a finite
dimensional subspace $\tilde V\subset L^2(\Sigma)$ with 
\[
\alpha \int_B k^2 |u^{(g)}_1|^2\dx 
\leq C \int_D  k^2 (q-1) |u^{(g)}_1|^2 \dx \quad \text{ for all } g\in \tilde V^\perp.
\]
However, this is contradicted by the localized potentials result in theorem~\ref{thm:localized_potentials}, which guarantees the existence of a
sequence $(g_j)_{j\in \N}\subset \tilde{V}^\perp$ with
\[
\int_{B} |u_{1}^{(g_j)}|^2 \dx \to \infty, \quad \text{ and } \quad \int_{D} |u_{1}^{(g_j)}|^2 \dx \to 0.  
\]
Hence, $\Lambda(q)-\Lambda(1)-\alpha T_B$ cannot have only finitely many negative
eigenvalues.
\hfill $\Box$

\emph{Proof of theorem \ref{thm:main2}.}
The proof is analogous to that of theorem \ref{thm:main1}. We state it for the sake of completeness.
Let
\[
q_\text{min}\leq  q(x) \leq  q_\text{max}<1 \quad \text{ for all } x\in D \text{ (a.e.)}
\]
If $B\subseteq D$, then by the monotonicity relation in remark~\ref{remark:othermonotonicity}, together with theorem~\ref{thm:useful}, we have that
\begin{align*}
\lefteqn{\int_{\Sigma} g \left( \Lambda(q) - \Lambda(1)\right) g \dx[s]}\\
& \leq_{d(1)}  \int_\Omega  k^2 (q-1) |u^{(g)}_q|^2 \dx
\leq - \int_D  k^2  (1-q_\text{max}) |u^{(g)}_q|^2 \dx\\
 &\leq -c (1-q_\text{max}) \int_D  k^2  |u^{(g)}_1|^2 \dx
  \leq -c (1-q_\text{max}) \int_B  k^2  |u^{(g)}_1|^2 \dx\\
&= -c (1-q_\text{max})\int_{\Sigma} g T_B g \dx[s],
\end{align*}
with a constant $c>0$ from theorem~\ref{thm:useful}. This shows that $B\subseteq D$ implies 
\[
 \alpha T_B \leq_{d(1)}  \Lambda(1) - \Lambda(q) \quad \text{ for all } \alpha\leq c (1-q_\text{max})=:\alpha_\text{max},
\]
so that (a) is proven.

To prove (b) by contradiction, let $B\not\subseteq D$, $\alpha>0$, and assume that
\begin{equation}\labeq{main2_hilf1}
\Lambda(1)-\Lambda(q) \geq_\text{fin} \alpha T_B.
\end{equation}
By the monotonicity relation in theorem \ref{thm:monotonicity}, we have that
\begin{equation}\labeq{main2_hilf2}
\int_{\Sigma} g \left( \Lambda(1) - \Lambda(q)\right) g \dx[s] \leq_\text{fin}
\int_D  k^2 (1-q) |u^{(g)}_1|^2 \dx.
\end{equation}
Combining \req{main2_hilf1} and \req{main2_hilf2} using the transitivity result from
lemma~\ref{lemma:order_transitive}, we have that 
\[
\alpha \int_B k^2 |u^{(g)}_1|^2\dx  \leq_\text{fin}  \int_D  k^2 (1-q) |u^{(g)}_1|^2 \dx.
\]
However, this is contradicted by theorem~\ref{thm:localized_potentials}, which guarantees (for each finite-dimensional space $V\subset L^2(\Sigma)$) the existence of a sequence $(g_j)_{j\in \N}\subset V^\perp$ with
\[
\int_{B} |u_{0}^{(g_j)}|^2 \dx \to \infty, \quad \text{ and } \quad \int_{D} |u_{0}^{(g_j)}|^2 \dx \to 0.  
\]
Hence, $\Lambda(1)-\Lambda(q)-\alpha T_B$ cannot have only finitely many negative
eigenvalues, which shows (b).
\hfill $\Box$

\subsection{Remarks and extensions}
\label{Subsec:remarks}
We finish this section with some remarks on possible extensions of our results.
Theorem \ref{thm:main1} and \ref{thm:main2} hold with analogous proofs also for the case that the homogeneous background scattering index is replaced by a known inhomogeneous function $q_0\in L^\infty(\Omega)$. Using the concept of the inner and outer support from \cite{harrach2013monotonicity} (see also \cite{kusiak2003scattering,gebauer2008factorization,harrach2010exact} for the origins of this concept), we can also treat the case where $\Omega\setminus \overline{D}$ is not connected or 
where there is no clear jump of the scattering index. The monotonicity tests will then determine $D$ up to the difference of the inner and outer support. Moreover, the so-called indefinite case that the domain contains scatterers with higher, and scatterers with lower refractive index, can be treated by shrinking a large test region analogously to \cite{harrach2013monotonicity}, see also \cite{garde2017regularized}.

\bibliographystyle{alpha}
\bibliography{literaturliste}

\newcommand{\etalchar}[1]{$^{#1}$}
\begin{thebibliography}{VMC{\etalchar{+}}17}

\bibitem[AH13]{arnold2013unique}
Lilian Arnold and Bastian Harrach.
\newblock Unique shape detection in transient eddy current problems.
\newblock {\em Inverse Problems}, 29(9):095004, 2013.

\bibitem[AIM09]{AstalaIwaniecMartin}
Kari Astala, Tadeusz Iwaniec, and Gaven Martin.
\newblock {\em Elliptic partial differential equations and quasiconformal
  mappings in the plane}, volume~48 of {\em Princeton Mathematical Series}.
\newblock Princeton University Press, Princeton, NJ, 2009.

\bibitem[Ale90]{alessandrini1990}
Giovanni Alessandrini.
\newblock Singular solutions of elliptic equations and the determination of
  conductivity by boundary measurements.
\newblock {\em J. Differential Equations}, 84(2):252--272, 1990.

\bibitem[Ale12]{Alessandrini_ucp}
Giovanni Alessandrini.
\newblock Strong unique continuation for general elliptic equations in 2{D}.
\newblock {\em J. Math. Anal. Appl.}, 386(2):669--676, 2012.

\bibitem[AP06]{astala2006calderon}
Kari Astala and Lassi P{\"a}iv{\"a}rinta.
\newblock Calder{\'o}n's inverse conductivity problem in the plane.
\newblock {\em Annals of Mathematics}, pages 265--299, 2006.

\bibitem[AU04]{AmmariUhlmann}
Habib Ammari and Gunther Uhlmann.
\newblock Reconstruction of the potential from partial {C}auchy data for the
  {S}chr\"odinger equation.
\newblock {\em Indiana Univ. Math. J.}, 53(1):169--183, 2004.

\bibitem[BH00]{Bru00}
Martin Br{\"u}hl and Martin Hanke.
\newblock Numerical implementation of two noniterative methods for locating
  inclusions by impedance tomography.
\newblock {\em Inverse Problems}, 16:1029--1042, 2000.

\bibitem[BHHM17]{barth2017detecting}
Andrea Barth, Bastian Harrach, Nuutti Hyv{\"o}nen, and Lauri Mustonen.
\newblock Detecting stochastic inclusions in electrical impedance tomography.
\newblock {\em Inverse Problems}, 33(11):115012, 2017.

\bibitem[BHKS18]{brander2018monotonicity}
Tommi Brander, Bastian Harrach, Manas Kar, and Mikko Salo.
\newblock Monotonicity and enclosure methods for the p-\uppercase{L}aplace
  equation.
\newblock {\em SIAM Journal on Applied Mathematics}, 78(2):742--758, 2018.

\bibitem[Br{\"u}01]{Bru01}
Martin Br{\"u}hl.
\newblock Explicit characterization of inclusions in electrical impedance
  tomography.
\newblock {\em SIAM J. Math. Anal.}, 32(6):1327--1341, 2001.

\bibitem[Cal80]{calderon1980inverse}
Alberto~P Calder\'on.
\newblock On an inverse boundary value problem.
\newblock In W~H Meyer and M~A Raupp, editors, {\em Seminar on Numerical
  Analysis and its Application to Continuum Physics}, pages 65--73. Brasil.
  Math. Soc., Rio de Janeiro, 1980.

\bibitem[Cal06]{calderon2006inverse}
Alberto~P Calder\'on.
\newblock On an inverse boundary value problem.
\newblock {\em Comput. Appl. Math.}, 25(2--3):133--138, 2006.

\bibitem[CR16]{caro2016global}
Pedro Caro and Keith~M Rogers.
\newblock Global uniqueness for the \uppercase{C}alder{\'o}n problem with
  \uppercase{L}ipschitz conductivities.
\newblock In {\em Forum of Mathematics, Pi}, volume~4. Cambridge University
  Press, 2016.

\bibitem[dFG92]{de1992strict}
Djairo~G de~Figueiredo and Jean-Pierre Gossez.
\newblock Strict monotonicity of eigenvalues and unique continuation.
\newblock {\em Comm. Partial Differential Equations}, 17(1-2):339--346, 1992.

\bibitem[Dru98]{druskin1998uniqueness}
Vladimir Druskin.
\newblock On the uniqueness of inverse problems from incomplete boundary data.
\newblock {\em SIAM Journal on Applied Mathematics}, 58(5):1591--1603, 1998.

\bibitem[Gar17]{garde2017comparison}
Henrik Garde.
\newblock Comparison of linear and non-linear monotononicity-based shape
  reconstruction using exact matrix characterizations.
\newblock {\em Inverse Problems in Science and Engineering}, 2017.

\bibitem[Geb08]{gebauer2008localized}
Bastian Gebauer.
\newblock Localized potentials in electrical impedance tomography.
\newblock {\em Inverse Probl. Imaging}, 2(2):251--269, 2008.

\bibitem[GH08]{gebauer2008factorization}
Bastian Gebauer and Nuutti Hyv{\"o}nen.
\newblock Factorization method and inclusions of mixed type in an inverse
  elliptic boundary value problem.
\newblock {\em Inverse Probl. Imaging}, 2(3):355--372, 2008.

\bibitem[GH18]{griesmaier2018monotonicity}
Roland Griesmaier and Bastian Harrach.
\newblock Monotonicity in inverse medium scattering on unbounded domains.
\newblock {\em SIAM J. Appl. Math.}, 78(5):2533--2557, 2018.

\bibitem[GS17]{garde2017convergence}
Henrik Garde and Stratos Staboulis.
\newblock Convergence and regularization for monotonicity-based shape
  reconstruction in electrical impedance tomography.
\newblock {\em Numerische Mathematik}, 135(4):1221--1251, 2017.

\bibitem[GS19]{garde2017regularized}
Henrik Garde and Stratos Staboulis.
\newblock The regularized monotonicity method: detecting irregular indefinite
  inclusions.
\newblock {\em Inverse Probl. Imag.}, 13(1):93--116, 2019.

\bibitem[GT13]{GT_survey}
Colin Guillarmou and Leo Tzou.
\newblock The {C}alder\'on inverse problem in two dimensions.
\newblock In {\em Inverse problems and applications: inside out. {II}},
  volume~60 of {\em Math. Sci. Res. Inst. Publ.}, pages 119--166. Cambridge
  Univ. Press, Cambridge, 2013.

\bibitem[Har09]{harrach2009uniqueness}
Bastian Harrach.
\newblock On uniqueness in diffuse optical tomography.
\newblock {\em Inverse Problems}, 25:055010 (14pp), 2009.

\bibitem[Har12]{harrach2012simultaneous}
Bastian Harrach.
\newblock Simultaneous determination of the diffusion and absorption
  coefficient from boundary data.
\newblock {\em Inverse Probl. Imaging}, 6(4):663--679, 2012.

\bibitem[Har13]{harrach2013recent}
Bastian Harrach.
\newblock Recent progress on the factorization method for electrical impedance
  tomography.
\newblock {\em Computational and mathematical methods in medicine}, 2013, 2013.

\bibitem[Har19]{harrach2018uniqueness}
Bastian Harrach.
\newblock Uniqueness and {L}ipschitz stability in electrical impedance
  tomography with finitely many electrodes.
\newblock {\em Inverse Problems}, 35(2):024005, 2019.

\bibitem[HL19a]{harrach2017monotonicity}
Bastian Harrach and Yi-Hsuan Lin.
\newblock Monotonicity-based inversion of the fractional {S}chr\"odinger
  equation {I}. {P}ositive potentials.
\newblock {\em SIAM J. Math. Anal.}, 51(4):3092--3111, 2019.

\bibitem[HL19b]{harrach2019fractional_II}
Bastian Harrach and Yi-Hsuan Lin.
\newblock Monotonicity-based inversion of the fractional {S}chr\"odinger
  equation {II}. {G}eneral potentials and stability.
\newblock {\em arXiv preprint arXiv:1903.08771}, 2019.

\bibitem[HLL18]{harrach2018localizing}
Bastian Harrach, Yi-Hsuan Lin, and Hongyu Liu.
\newblock On localizing and concentrating electromagnetic fields.
\newblock {\em SIAM J. Appl. Math.}, 78(5):2558--2574, 2018.

\bibitem[HLU15]{harrach2015combining}
Bastian Harrach, Eunjung Lee, and Marcel Ullrich.
\newblock Combining frequency-difference and ultrasound modulated electrical
  impedance tomography.
\newblock {\em Inverse Problems}, 31(9):095003, 2015.

\bibitem[HM16]{harrach2016enhancing}
Bastian Harrach and Mach~Nguyet Minh.
\newblock Enhancing residual-based techniques with shape reconstruction
  features in electrical impedance tomography.
\newblock {\em Inverse Problems}, 32(12):125002, 2016.

\bibitem[HM18]{harrach2018monotonicity}
Bastian Harrach and Mach~Nguyet Minh.
\newblock Monotonicity-based regularization for phantom experiment data in
  electrical impedance tomography.
\newblock In {\em New Trends in Parameter Identification for Mathematical
  Models}, pages 107--120. Springer, 2018.

\bibitem[HM19]{harrach2018global}
Bastian Harrach and Houcine Meftahi.
\newblock Global uniqueness and {L}ipschitz-stability for the inverse {R}obin
  transmission problem.
\newblock {\em SIAM J. Appl. Math.}, 79(2):525--550, 2019.

\bibitem[H{\"o}r94]{hormander1983analysis}
Lars H{\"o}rmander.
\newblock {\em The analysis of linear partial differential operators {III}},
  volume 274.
\newblock Springer, 1994.

\bibitem[HPS17]{harrach2017oberwolfach}
Bastian Harrach, Valter Pohjola, and Mikko Salo.
\newblock The monotonicity method for inverse scattering.
\newblock In {\em Oberwolfach Report}, volume~24, pages 57--60, 2017.

\bibitem[HPS19]{harrach2019dimension}
Bastian Harrach, Valter Pohjola, and Mikko Salo.
\newblock Dimension bounds in monotonicity methods for the {H}elmholtz
  equation.
\newblock {\em SIAM J. Math. Anal.}, 51(4):2995--3019, 2019.

\bibitem[HS10]{harrach2010exact}
Bastian Harrach and Jin~Keun Seo.
\newblock Exact shape-reconstruction by one-step linearization in electrical
  impedance tomography.
\newblock {\em SIAM Journal on Mathematical Analysis}, 42(4):1505--1518, 2010.

\bibitem[HT01]{hadi2001strong}
Islam~Eddine Hadi and N~Tsouli.
\newblock Strong unique continuation of eigenfunctions for
  p-\uppercase{L}aplacian operator.
\newblock {\em Int. J. Math. Math. Sci.}, 25(3):213--216, 2001.

\bibitem[HT13]{haberman2013uniqueness}
Boaz Haberman and Daniel Tataru.
\newblock Uniqueness in \uppercase{C}alder{\'o}n's problem with
  \uppercase{L}ipschitz conductivities.
\newblock {\em Duke Mathematical Journal}, 162(3):497--516, 2013.

\bibitem[HU13]{harrach2013monotonicity}
Bastian Harrach and Marcel Ullrich.
\newblock Monotonicity-based shape reconstruction in electrical impedance
  tomography.
\newblock {\em SIAM Journal on Mathematical Analysis}, 45(6):3382--3403, 2013.

\bibitem[HU15]{harrach2015resolution}
Bastian Harrach and Marcel Ullrich.
\newblock Resolution guarantees in electrical impedance tomography.
\newblock {\em IEEE Trans. Med. Imaging}, 34:1513--1521, 2015.

\bibitem[HU17]{harrach2017local}
Bastian Harrach and Marcel Ullrich.
\newblock Local uniqueness for an inverse boundary value problem with partial
  data.
\newblock {\em Proceedings of the American Mathematical Society},
  145(3):1087--1095, 2017.

\bibitem[Ike98]{ikehata1998size}
Masaru Ikehata.
\newblock Size estimation of inclusion.
\newblock {\em J. Inverse Ill-Posed Probl.}, 6(2):127--140, 1998.

\bibitem[Ike99]{ikehata1999identification}
Masaru Ikehata.
\newblock Identification of the shape of the inclusion having essentially
  bounded conductivity.
\newblock {\em Journal of Inverse and Ill-Posed Problems}, 7(6):533--540, 1999.

\bibitem[Isa88]{isakov1988}
Victor Isakov.
\newblock On uniqueness of recovery of a discontinuous conductivity
  coefficient.
\newblock {\em Comm. Pure Appl. Math.}, 41(7):865--877, 1988.

\bibitem[Isa07]{Isakov}
Victor Isakov.
\newblock On uniqueness in the inverse conductivity problem with local data.
\newblock {\em Inverse Probl. Imaging}, 1(1):95--105, 2007.

\bibitem[IUY10]{IUY}
Oleg~Yu. Imanuvilov, Gunther Uhlmann, and Masahiro Yamamoto.
\newblock The {C}alder\'on problem with partial data in two dimensions.
\newblock {\em J. Amer. Math. Soc.}, 23(3):655--691, 2010.

\bibitem[IUY15]{imanuvilov2015neumann}
O~Yu Imanuvilov, Gunther Uhlmann, and Masahiro Yamamoto.
\newblock The {N}eumann-to-{D}irichlet map in two dimensions.
\newblock {\em Advances in Mathematics}, 281:578--593, 2015.

\bibitem[KG08]{kirsch2008factorization}
Andreas Kirsch and Natalia Grinberg.
\newblock {\em The factorization method for inverse problems}, volume~36.
\newblock Oxford University Press, 2008.

\bibitem[Kir98]{Kir98}
Andreas Kirsch.
\newblock Characterization of the shape of a scattering obstacle using the
  spectral data of the far field operator.
\newblock {\em Inverse Problems}, 14(6):1489--1512, 1998.

\bibitem[KS03]{kusiak2003scattering}
Steven Kusiak and John Sylvester.
\newblock The scattering support.
\newblock {\em Communications on Pure and Applied Mathematics},
  56(11):1525--1548, 2003.

\bibitem[KS13]{KS}
Carlos Kenig and Mikko Salo.
\newblock The {C}alder\'on problem with partial data on manifolds and
  applications.
\newblock {\em Anal. PDE}, 6(8):2003--2048, 2013.

\bibitem[KS14]{kenig2014recent}
Carlos Kenig and Mikko Salo.
\newblock Recent progress in the \uppercase{C}alder{\'o}n problem with partial
  data.
\newblock {\em Contemp. Math}, 615:193--222, 2014.

\bibitem[KSS97]{kang1997inverse}
Hyeonbae Kang, Jin~Keun Seo, and Dongwoo Sheen.
\newblock The inverse conductivity problem with one measurement: stability and
  estimation of size.
\newblock {\em SIAM J. Math. Anal.}, 28(6):1389--1405, 1997.

\bibitem[KSU07]{KSU}
Carlos~E. Kenig, Johannes Sj\"ostrand, and Gunther Uhlmann.
\newblock The {C}alder\'on problem with partial data.
\newblock {\em Ann. of Math. (2)}, 165(2):567--591, 2007.

\bibitem[KU16]{krupchyk2016calderon}
Katya Krupchyk and Gunther Uhlmann.
\newblock The \uppercase{C}alder{\'o}n problem with partial data for
  conductivities with 3/2 derivatives.
\newblock {\em Communications in Mathematical Physics}, 348(1):185--219, 2016.

\bibitem[KV84]{kohn1984determining}
Robert~V Kohn and Michael Vogelius.
\newblock Determining conductivity by boundary measurements.
\newblock {\em Communications on Pure and Applied Mathematics}, 37(3):289--298,
  1984.

\bibitem[KV85]{kohn1985determining}
Robert~V Kohn and Michael Vogelius.
\newblock Determining conductivity by boundary measurements \uppercase{II}.
  {I}nterior results.
\newblock {\em Communications on Pure and Applied Mathematics}, 38(5):643--667,
  1985.

\bibitem[Lax02]{PL}
Peter~D Lax.
\newblock {\em Functional analysis}.
\newblock Wiley-Interscience, New York, 2002.

\bibitem[MVVT16]{maffucci2016novel}
Antonio Maffucci, Antonio Vento, Salvatore Ventre, and Antonello Tamburrino.
\newblock A novel technique for evaluating the effective permittivity of
  inhomogeneous interconnects based on the monotonicity property.
\newblock {\em IEEE Transactions on Components, Packaging and Manufacturing
  Technology}, 6(9):1417--1427, 2016.

\bibitem[Nac96]{nachman1996global}
Adrian~I Nachman.
\newblock Global uniqueness for a two-dimensional inverse boundary value
  problem.
\newblock {\em Ann. of Math. (2)}, 143(1):71--96, 1996.

\bibitem[Reg01]{regbaoui2001unique}
Rachid Regbaoui.
\newblock Unique continuation from sets of positive measure.
\newblock In {\em Carleman Estimates and Applications to Uniqueness and Control
  Theory}, pages 179--190. Springer, 2001.

\bibitem[RS72]{RSI}
Michael Reed and Barry Simon.
\newblock {\em Methods of modern mathematical physics, Volume I: Functional
  analysis}.
\newblock Academic Press, San Diego, 1972.

\bibitem[SU87]{sylvester1987global}
John Sylvester and Gunther Uhlmann.
\newblock A global uniqueness theorem for an inverse boundary value problem.
\newblock {\em Annals of mathematics}, pages 153--169, 1987.

\bibitem[SUG{\etalchar{+}}17]{su2017monotonicity}
Zhiyi Su, Lalita Udpa, Gaspare Giovinco, Salvatore Ventre, and Antonello
  Tamburrino.
\newblock Monotonicity principle in pulsed eddy current testing and its
  application to defect sizing.
\newblock In {\em Applied Computational Electromagnetics Society
  Symposium-Italy (ACES), 2017 International}, pages 1--2. IEEE, 2017.

\bibitem[TR02]{tamburrino2002new}
Antonello Tamburrino and Guglielmo Rubinacci.
\newblock A new non-iterative inversion method for electrical resistance
  tomography.
\newblock {\em Inverse Problems}, 18(6):1809, 2002.

\bibitem[TSV{\etalchar{+}}16]{tamburrino2016monotonicity}
Antonello Tamburrino, Zhiyi Sua, Salvatore Ventre, Lalita Udpa, and Satish~S
  Udpa.
\newblock Monotonicity based imang method in time domain eddy current testing.
\newblock {\em Electromagnetic Nondestructive Evaluation (XIX)}, 41:1, 2016.

\bibitem[VMC{\etalchar{+}}17]{ventre2017design}
Salvatore Ventre, Antonio Maffucci, Fran{\c{c}}ois Caire, Nechtan Le~Lostec,
  Antea Perrotta, Guglielmo Rubinacci, Bernard Sartre, Antonio Vento, and
  Antonello Tamburrino.
\newblock Design of a real-time eddy current tomography system.
\newblock {\em IEEE Transactions on Magnetics}, 53(3):1--8, 2017.

\bibitem[ZHS18]{zhou2018monotonicity}
Liangdong Zhou, Bastian Harrach, and Jin~Keun Seo.
\newblock Monotonicity-based electrical impedance tomography for lung imaging.
\newblock {\em Inverse Problems}, 34(4):045005, 2018.

\end{thebibliography}

\end{document}